\documentclass{aic}
\newtheorem{question}{Question}[section]
\newtheorem{lemma}[question]{Lemma}
\newtheorem{theorem}[question]{Theorem}
\newtheorem{conjecture}[question]{Conjecture}
\newtheorem{corollary}[question]{Corollary}

\usepackage[T1]{fontenc}

\makeatletter
\newcommand{\leqnomode}{\tagsleft@true}
\newcommand{\reqnomode}{\tagsleft@false}
\makeatother

\usepackage{latexsym}
\usepackage{amsfonts}

\usepackage{tikz}
\usepackage{float}
\usepackage{thm-restate}

\newcommand{\vsp}{\medskip}
\usepackage{enumerate}

\newcommand*\samethanks[1][\value{footnote}]{\footnotemark[#1]}

\def\dd{\hbox{-}}

\DeclareMathOperator{\tw}{tw}
\DeclareMathOperator{\width}{width}

\DeclareMathOperator{\cl}{cl}
\DeclareMathOperator{\Hub}{Hub}
\DeclareMathOperator{\hdim}{hdim}
\DeclareMathOperator{\Core}{Core}

\newcounter{tbox}

\newcommand{\sta}[1]{\vspace*{0.3cm}\refstepcounter{tbox}\noindent{ \parbox{\textwidth}{(\thetbox) \emph{#1}}}\vspace*{0.3cm}}

\makeatletter
\newcommand{\otherlabel}[2]{\protected@edef\@currentlabel{#2}\label{#1}}
\makeatother

\mathchardef\mh="2D

\aicAUTHORdetails{%
  title = {Induced subgraphs and tree decompositions\\
III. Three-path-configurations and logarithmic treewidth}, 
  author = {Tara Abrishami, Maria Chudnovsky, Sepehr Hajebi, Sophie Spirkl},
  plaintextauthor = {Tara Abrishami, Maria Chudnovsky, Sepehr Hajebi, Sophie Spirkl},
  runningtitle = {Induced subgraphs and tree decompositions III}, 
}   

\aicEDITORdetails{%
   year={2022},
   number={6},
   received={3 September 2021},   
   published={9 September 2022},  
   doi={10.19086/aic.2022.6},      
}   

\begin{document}

\begin{frontmatter}[classification=text]


\author[ta]{Tara Abrishami\thanks{Supported by NSF Grant DMS-1763817 and NSF-EPSRC Grant DMS-2120644.}}
\author[mc]{Maria Chudnovsky\samethanks}
\author[sh]{Sepehr Hajebi}
\author[sophie]{Sophie Spirkl\thanks{We acknowledge the support of the Natural Sciences and Engineering Research Council of Canada (NSERC), [funding reference number RGPIN-2020-03912].
Cette recherche a \'et\'e financ\'ee par le Conseil de recherches en sciences naturelles et en g\'enie du Canada (CRSNG), [num\'ero de r\'ef\'erence RGPIN-2020-03912].}}

\begin{abstract}
A \textit{theta} is a graph consisting of two non-adjacent vertices and three internally disjoint paths between them, each of length at least two.  For a family $\mathcal{H}$ of graphs, we say a graph $G$ is $\mathcal{H}$-\textit{free} if no induced subgraph of $G$ is isomorphic to a member of $\mathcal{H}$. We prove a conjecture of Sintiari and Trotignon, that there exists an absolute constant $c$ for which every (theta, triangle)-free graph $G$ has treewidth at most $c\log (|V(G)|)$. A construction by Sintiari and Trotignon shows that this bound is asymptotically best possible, and (theta, triangle)-free graphs comprise the first known hereditary class of graphs with arbitrarily large yet logarithmic treewidth.

Our main result is in fact a  generalization of the above conjecture, that treewidth is at most logarithmic in $|V(G)|$ for every graph $G$ excluding the so-called \textit{three-path-configurations} as well as a fixed complete graph. It follows that several \textsf{NP}-hard problems such as \textsc{Stable Set}, \textsc{Vertex Cover}, \textsc{Dominating Set} and \textsc{Coloring} admit polynomial time algorithms in graphs excluding the three-path-configurations and a fixed complete graph.
\end{abstract}
\end{frontmatter}

\section{Introduction}
All graphs in this paper are finite and simple. Let $G = (V(G),E(G))$ be a graph. For a set $X \subseteq V(G)$ we denote by $G[X]$ the subgraph of $G$ induced by $X$. For $X \subseteq V(G)$, $G \setminus X$ denotes the subgraph induced by $V(G) \setminus X$. In this paper, we use induced subgraphs and their vertex sets interchangeably. Let $v \in V(G)$. The \emph{open neighborhood of $v$}, denoted by $N(v)$, is the set of all vertices in $V(G)$ adjacent to $v$. The \emph{closed neighborhood of $v$}, denoted by $N[v]$, is $N(v) \cup \{v\}$. Let $X \subseteq V(G)$. The \emph{open neighborhood of $X$}, denoted by $N(X)$, is the set of all vertices in $V(G) \setminus X$ with at least one neighbor in $X$. The \emph{closed neighborhood of $X$}, denoted by $N[X]$, is $N(X) \cup X$. If $H$ is an induced subgraph of $G$ and $X \subseteq V(G)$, then $N_H(X)=N(X) \cap H$ and $N_H[X]=N_H(X) \cup X$.
Let $Y \subseteq V(G)$ be disjoint from $X$. We say $X$ is \textit{complete} to $Y$ if all edges with an end in $X$ and an end in $Y$ are present in $G$, and $X$ is \emph{anticomplete}
to $Y$ if there are no edges between $X$ and $Y$.

For a graph $G = (V(G),E(G))$, a \emph{tree decomposition} $(T, \chi)$ of $G$ consists of a tree $T$ and a map $\chi: V(T) \to 2^{V(G)}$ with the following properties: 
\begin{enumerate}[(i)]
\itemsep -.2em
    \item For every $v \in V(G)$, there exists $t \in V(T)$ such that $v \in \chi(t)$. 
    
    \item For every $v_1v_2 \in E(G)$, there exists $t \in V(T)$ such that $v_1, v_2 \in \chi(t)$.
    
    \item For every $v \in V(G)$, the subgraph of $T$ induced by $\{t \in V(T) \mid v \in \chi(t)\}$ is connected.
\end{enumerate}

For each $t\in V(T)$, we refer to $\chi(t)$ as a \textit{bag of} $(T, \chi)$.  The \emph{width} of a tree decomposition $(T, \chi)$, denoted by $\width(T, \chi)$, is $\max_{t \in V(T)} |\chi(t)|-1$. The \emph{treewidth} of $G$, denoted by $\tw(G)$, is the minimum width of a tree decomposition of $G$. 

Treewidth, first introduced by Robertson and Seymour in their monumental work on graph minors, is an extensively studied graph parameter, mostly due to the fact that graphs of bounded treewidth exhibit interesting structural
\cite{RS-GMV} and algorithmic \cite{Bodlaender1988DynamicTreewidth} properties. Accordingly, one would naturally desire to understand the structure of graphs with large treewidth, and in particular the unavoidable substructures emerging in them. For instance, for each $k$, the {\em $(k \times k)$-wall}, denoted by $W_{k \times k}$, is a planar graph with maximum degree three and with treewidth $k$ (see Figure~\ref{fig:5x5wall}; a precise definition can be found in \cite{wallpaper}). Every subdivision of $W_{k \times k}$ is also a
graph of treewidth $k$. The unavoidable subgraphs of graphs with large treewidth are fully characterized by the Grid Theorem of Robertson and Seymour, the following.
\begin{theorem}[\cite{RS-GMV}]\label{wallminor}
For every integer $t$ there exists $c=c(t)$ such that 
every graph of treewidth at least $c$
contains a subdivision of $W_{t \times t}$ as a subgraph.
\end{theorem}

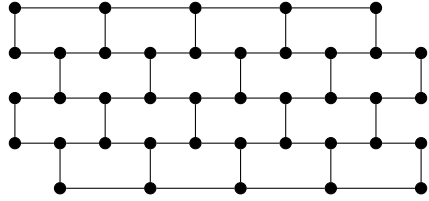
\begin{figure}
\centering

\begin{tikzpicture}[scale=2,auto=left]
\tikzstyle{every node}=[inner sep=1.5pt, fill=black,circle,draw]  
\centering

\node (s10) at (0,1.2) {};
\node(s12) at (0.6,1.2){};
\node(s14) at (1.2,1.2){};
\node(s16) at (1.8,1.2){};
\node(s18) at (2.4,1.2){};

\node (s20) at (0,0.9) {};
\node (s21) at (0.3,0.9) {};
\node(s22) at (0.6,0.9){};
\node (s23) at (0.9,0.9) {};
\node(s24) at (1.2,0.9){};
\node (s25) at (1.5,0.9) {};
\node(s26) at (1.8,0.9){};
\node (s27) at (2.1,0.9) {};
\node(s28) at (2.4,0.9){};
\node (s29) at (2.7,0.9) {};

\node (s30) at (0,0.6) {};
\node (s31) at (0.3,0.6) {};
\node(s32) at (0.6,0.6){};
\node (s33) at (0.9,0.6) {};
\node(s34) at (1.2,0.6){};
\node (s35) at (1.5,0.6) {};
\node(s36) at (1.8,0.6){};
\node (s37) at (2.1,0.6) {};
\node(s38) at (2.4,0.6){};
\node (s39) at (2.7,0.6) {};

\node (s40) at (0,0.3) {};
\node (s41) at (0.3,0.3) {};
\node(s42) at (0.6,0.3){};
\node (s43) at (0.9,0.3) {};
\node(s44) at (1.2,0.3){};
\node (s45) at (1.5,0.3) {};
\node(s46) at (1.8,0.3){};
\node (s47) at (2.1,0.3) {};
\node(s48) at (2.4,0.3) {};
\node (s49) at (2.7,0.3) {};

\node (s51) at (0.3,0.0) {};
\node (s53) at (0.9,0.0) {};
\node (s55) at (1.5,0.0) {};
\node (s57) at (2.1,0.0) {};
\node (s59) at (2.7,0.0) {};

\foreach \from/\to in {s10/s12, s12/s14,s14/s16,s16/s18}
\draw [-] (\from) -- (\to);

\foreach \from/\to in {s20/s21, s21/s22, s22/s23, s23/s24, s24/s25, s25/s26,s26/s27,s27/s28,s28/s29}
\draw [-] (\from) -- (\to);

\foreach \from/\to in {s30/s31, s31/s32, s32/s33, s33/s34, s34/s35, s35/s36,s36/s37,s37/s38,s38/s39}
\draw [-] (\from) -- (\to);

\foreach \from/\to in {s40/s41, s41/s42, s42/s43, s43/s44, s44/s45, s45/s46,s46/s47,s47/s48,s48/s49}
\draw [-] (\from) -- (\to);

\foreach \from/\to in {s51/s53, s53/s55,s55/s57,s57/s59}
\draw [-] (\from) -- (\to);

\foreach \from/\to in {s10/s20, s30/s40}
\draw [-] (\from) -- (\to);

\foreach \from/\to in {s21/s31,s41/s51}
\draw [-] (\from) -- (\to);

\foreach \from/\to in {s12/s22, s32/s42}
\draw [-] (\from) -- (\to);

\foreach \from/\to in {s23/s33,s43/s53}
\draw [-] (\from) -- (\to);

\foreach \from/\to in {s14/s24, s34/s44}
\draw [-] (\from) -- (\to);

\foreach \from/\to in {s25/s35,s45/s55}
\draw [-] (\from) -- (\to);

\foreach \from/\to in {s16/s26,s36/s46}
\draw [-] (\from) -- (\to);

\foreach \from/\to in {s27/s37,s47/s57}
\draw [-] (\from) -- (\to);

\foreach \from/\to in {s18/s28,s38/s48}
\draw [-] (\from) -- (\to);

\foreach \from/\to in {s29/s39,s49/s59}
\draw [-] (\from) -- (\to);

\end{tikzpicture}

\caption{$W_{5 \times 5}$}
\label{fig:5x5wall}
\end{figure}

Following the same line of thought, our motivation is to study the unavoidable induced subgraphs of graphs with large treewidth. Together with subdivided walls mentioned above, complete graphs and complete bipartite graphs are easily observed to have arbitrarily large treewidth: the complete graph $K_{t+1}$ and the complete bipartite graph $K_{t,t}$ both have treewidth $t$. Line graphs of subdivided walls form another family of graphs with unbounded treewidth, where the {\em line graph} $L(F)$ of a graph $F$ is the graph with vertex set $E(G)$, such that two vertices of $L(F)$ are adjacent if the corresponding edges of $G$ share an end.  One may ask whether these graphs are all we have to exclude as induced subgraphs to obtain a constant bound on the treewidth:

\begin{question}\label{usefulQ}
Is it true that for all $t$, there exists $c=c(t)$ such that every graph with treewidth more than $c$ contains as an induced subgraph either a $K_t$, or a $K_{t,t}$, or a subdivision of $W_{t \times t}$ or the line graph of a subdivision of $W_{t \times t}$?
\end{question}

Sintiari and Trotignon \cite{mainconj} provided a negative answer to this question. To describe their result, we require a few more definitions.  Let $H$ be a graph. We say $G$ \emph{contains} $H$ if $G$ has an induced subgraph isomorphic to $H$. We say $G$ is \emph{$H$-free} if $G$ does not contain $H$.  For a family $\mathcal{H}$ of graphs we say that $G$ is $\mathcal{H}$-free if $G$ is $H$-free for every $H \in \mathcal{H}$. Given a graph $G$, a {\em path in $G$} is an induced subgraph of $G$ that is a path. If $P$ is a path in $G$, we write $P = p_1 \dd \cdots \dd p_k$ to mean that $V(P) = \{p_1, \dots, p_k\}$, and $p_i$ is adjacent to $p_j$ if and only if $|i-j| = 1$. We call the vertices $p_1$ and $p_k$ the \emph{ends of $P$}, and say that $P$ is \emph{from $p_1$ to $p_k$}. The \emph{interior of $P$}, denoted by $P^*$, is the set $V(P) \setminus \{p_1, p_k\}$. The \emph{length} of a path $P$ is the number of edges in $P$.
 
A {\em theta} is a graph consisting of two non-adjacent vertices $a, b$ and three paths $P_1, P_2, P_3$ from $a$ to $b$ of length at least two, such that $P_1^*, P_2^*, P_3^*$ are mutually disjoint and anticomplete to each other. If a graph $G$ contains an induced subgraph $H$ which is a theta, and $a, b$ are the two vertices of degree three in $H$, then we say that $G$ contains a theta \emph{between $a$ and $b$}. Note that the complete bipartite graph $K_{2,3}$ is a theta. Also, it is readily seen that for large enough $k$, all subdivisions of $W_{k\times k}$ contain thetas, and of course line graphs of subdivisions of $W_{k\times k}$ contain triangles. So the following theorem provides a negative answer to Question~\ref{usefulQ}.

\begin{theorem}[\cite{mainconj}]
\label{thm:layered_wheel}
For every integer $\ell \geq 1$, there exists a (theta, triangle)-free graph $G_{\ell}$ such that $\tw(G_\ell) \geq \ell$. 
\end{theorem}

 The authors of \cite{mainconj} observed that the number of vertices of the graphs $G_{\ell}$ from  Theorem~\ref{thm:layered_wheel} is exponential in their treewidth, while for walls and their line graphs, the number of vertices is polynomial in the treewidth. This radical difference leads to the following conjecture. 

\begin{conjecture}[\cite{mainconj}]
\label{mainconj}
There exists a constant $c$ such that if $G$ is a (theta, triangle)-free graph,
then $\tw(G)\leq c \log(|V (G)|)$. 
\end{conjecture}

We prove this conjecture. Indeed, our main result is a substantial generalization of Conjecture~\ref{mainconj} involving the so-called \textit{three-path-configurations}, which we define next. A {\em hole} in a graph is an induced cycle of length at least four. (The {\em length} of a hole is the number of vertices in it.)
 
A {\em pyramid} is a graph consisting of a vertex $a$ and a triangle $\{b_1, b_2, b_3\}$, and three paths $P_i$ from $a$ to $b_i$ for $1 \leq i \leq 3$, all of length at least one, such that for $i \neq j$, the only edge between $P_i \setminus \{a\}$ 
and $P_j \setminus \{a\}$  is $b_ib_j$, and  at most one of $P_1, P_2, P_3$ has
length exactly one. We say $a$ is the {\em apex} of the pyramid and $b_1b_2b_3$ is the {\em base} of the pyramid. 

A {\em prism} is a graph consisting of two triangles $\{a_1,a_2,a_3\}$ and  $\{b_1, b_2, b_3\}$, and three paths $P_i$ from $a_i$ to $b_i$ for $1 \leq i \leq 3$,
all of length at least one, and such that for $i \neq j$ the only edges between
$P_i$ and $P_j$ are $a_ia_j$ and $b_ib_j$.
A {\em pinched prism} is a graph consisting of a hole $H$
of length at least six, together with a vertex $b_1$ such that $N_H(b_1)$
is an induced two-edge matching. We call $b_1$ the {\em center} of the pinched prism. (This graph is often called a `line wheel', but
here we choose to emphasize its similarity to a prism with a ``pinched'' path).
A {\em generalized prism} is a graph that is either a prism or a pinched prism. 
\begin{figure}[ht]
\label{fig:3PCs}
\begin{center}
\begin{tikzpicture}[scale=0.29]

\node[inner sep=2.5pt, fill=black, circle] at (0, 3)(v1){}; 
\node[inner sep=2.5pt, fill=black, circle] at (-3, 0)(v2){}; 
\node[inner sep=2.5pt, fill=black, circle] at (3, 0)(v3){}; 
\node[inner sep=2.5pt, fill=black, circle] at (0, 0)(v4){}; 
\node[inner sep=2.5pt, fill=black, circle] at (0, -6)(v5){};

\node[inner sep=2.5pt, fill=white, circle] at (0, -4.8)(v21){};

\draw[black, thick] (v1) -- (v2);
\draw[black, thick] (v1) -- (v3);
\draw[black, thick] (v1) -- (v4);
\draw[black, dotted, thick] (v2) -- (v5);
\draw[black, dotted, thick] (v3) -- (v5);
\draw[black, dotted, thick] (v4) -- (v5);

\end{tikzpicture}
\hspace{0.7cm}
\begin{tikzpicture}[scale=0.29]

\node[inner sep=2.5pt, fill=black, circle] at (0, 3)(v1){}; 
\node[inner sep=2.5pt, fill=black, circle] at (-3, 0)(v2){}; 
\node[inner sep=2.5pt, fill=black, circle] at (3, 0)(v3){}; 
\node[inner sep=2.5pt, fill=black, circle] at (-3, -6)(v4){}; 
\node[inner sep=2.5pt, fill=black, circle] at (0, -3)(v5){}; 
\node[inner sep=2.5pt, fill=black, circle] at (3, -6)(v6){};

\node[inner sep=2.5pt, fill=white, circle] at (0, -4.8)(v21){}; 

\draw[black, thick] (v1) -- (v2);
\draw[black, thick] (v1) -- (v3);
\draw[black, dotted, thick] (v1) -- (v5);
\draw[black, dotted, thick] (v2) -- (v4);
\draw[black, dotted, thick] (v3) -- (v6);
\draw[black, thick] (v4) -- (v5);
\draw[black, thick] (v4) -- (v6);
\draw[black, thick] (v5) -- (v6);

\end{tikzpicture}
\hspace{0.7cm}
\begin{tikzpicture}[scale=0.29]

\node[inner sep=2.5pt, fill=black, circle] at (-3, 3)(v1){}; 
\node[inner sep=2.5pt, fill=black, circle] at (0, 0)(v2){}; 
\node[inner sep=2.5pt, fill=black, circle] at (3, 3)(v3){}; 
\node[inner sep=2.5pt, fill=black, circle] at (-3, -6)(v4){}; 
\node[inner sep=2.5pt, fill=black, circle] at (0, -3)(v5){}; 
\node[inner sep=2.5pt, fill=black, circle] at (3, -6)(v6){};

\node[inner sep=2.5pt, fill=white, circle] at (0, -4.8)(v21){}; 

\draw[black, thick] (v1) -- (v2);
\draw[black, thick] (v1) -- (v3);
\draw[black, thick] (v2) -- (v3);
\draw[black, dotted, thick] (v1) -- (v4);
\draw[black, dotted, thick] (v2) -- (v5);
\draw[black, dotted, thick] (v3) -- (v6);
\draw[black, thick] (v4) -- (v5);
\draw[black, thick] (v4) -- (v6);
\draw[black, thick] (v5) -- (v6);

\end{tikzpicture}
\hspace{0.7cm}
\begin{tikzpicture}[scale=0.29]

\node[inner sep=2.5pt, fill=black, circle] at (-3, 3)(v1){}; 
\node[inner sep=2.5pt, fill=black, circle] at (0, -1.5)(v2){}; 
\node[inner sep=2.5pt, fill=black, circle] at (3, 3)(v3){}; 
\node[inner sep=2.5pt, fill=black, circle] at (-3, -6)(v4){}; 
\node[inner sep=2.5pt, fill=black, circle] at (3, -6)(v6){};

\node[inner sep=2.5pt, fill=white, circle] at (0, -4.8)(v21){}; 

\draw[black, thick] (v1) -- (v2);
\draw[black, thick] (v1) -- (v3);
\draw[black, thick] (v2) -- (v3);
\draw[black, dotted, thick] (v1) -- (v4);
\draw[black, dotted, thick] (v3) -- (v6);
\draw[black, thick] (v4) -- (v2);
\draw[black, thick] (v4) -- (v6);
\draw[black, thick] (v2) -- (v6);

\end{tikzpicture}
\end{center}
\caption{Theta, pyramid, prism, and pinched prism. The dotted lines represent paths of length at least one. In the pinched prism, the dotted lines represent paths of length at least two.}
\label{fig:forbidden_isgs}
\end{figure}
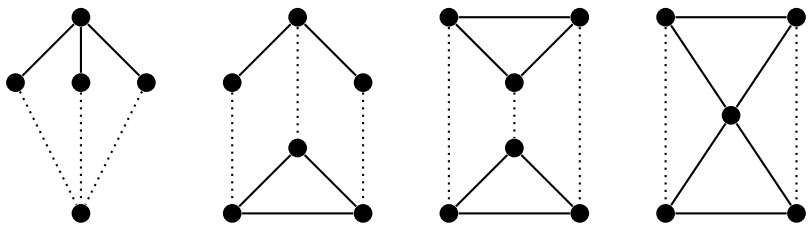

Finally, a graph is a {\em three-path-configuraton} if it is a theta, a prims, or a pyramid.
Let $\mathcal{C}$ be the class of (theta, pyramid, generalized prism)-free graphs. Also, for every integer $t \geq 1$, let  $\mathcal{C}_t$ be the class of (theta, pyramid, generalized prism, $K_t$)-free graphs.  Our main result is the following.

\begin{theorem}\label{mainthm}
  For every $t \geq 1$ there exists a constant $c_t$ such that every
  $G \in \mathcal{C}_t$ has treewidth at most $c_t \log(|V(G)|)$.
\end{theorem}

A \emph{clique cutset} in a graph $G$ is a clique $K$ of $G$ such that $G \setminus K$ is not connected. A special case of Lemma 3.1 from \cite{cliquetw} shows that clique cutsets do not affect treewidth:

\begin{theorem}
  \label{nocliquecutset}
  For every graph $G$ there exists an induced subgraphs $G'$ of $G$ such that
  $G'$ has no clique cutset, and $\tw(G)=\tw(G')$.
  \end{theorem}

It follows that in order to prove Theorem~\ref{mainthm}, it is enough to prove the following:
\begin{theorem}\label{mainthmnoclique}
  For every $t\geq $ there exists a constant $c_t$ such that every
  graph in $G \in \mathcal{C}_t$ with no clique cutset has treewidth at most $c_t \log(|V(G)|)$.
\end{theorem}

Note that since $\mathcal{C}_t$ contains all (theta, triangle)-free graphs for each $t \geq 3$, Theorem~\ref{mainthm} settles Conjecture~\ref{mainconj}. As discussed above, the construction from Theorem~\ref{thm:layered_wheel} shows that the bound provided by Theorem~\ref{mainthm} is asymptotically best possible. In fact, Theorem~\ref{mainthm} is the first result establishing a logarithmic bound on the treewidth in a hereditary (that is, closed under isomorphism and taking induced subgraphs) class of graphs. This is remarkable, especially because a considerable number of algorithmic advantages of bounded treewidth are still accessible in graphs of logarithmic treewidth. We elaborate on this in Section~\ref{algsec}. 

The proof of Theorem~\ref{mainthm} builds on a method developed in two previous papers of this series \cite{wallpaper, ACV} to bound the treewidth of graph classes with bounded maximum degree. Incidentally, unlike subdivided walls and their line graphs, the graphs $G_{\ell}$ from Theorem~\ref{thm:layered_wheel} contain vertices of arbitrarily large degree. So the following question is asked in \cite{septotree}:

\begin{question}[\cite{septotree}]\label{thetadegq}
Is it true that for every $\Delta>0$, there exists $c=c(\Delta)$ such that for every (theta,triangle)-free graph $G$ of maximum degree at most $\Delta$, we have $\tw(G)\leq c$?
\end{question}
In \cite{wallpaper}, with Dibek, Rz\k{a}\.{z}ewski and Vu\v{s}kovi\'c, we gave an affirmative answer to this question. More generally, it is also conjectured in \cite{Aboulker2020OnGraphs} that there is an affirmative answer to Question~\ref{usefulQ} restricted to graphs of bounded maximum degree. 

\begin{conjecture}[\cite{Aboulker2020OnGraphs}]
\label{conj:wall_bdd}
For all $k,\Delta > 0$, there exists $c=c(k,\delta)$ such that every graph with maximum degree at most $\Delta$ and treewidth more than $c$ contains a subdivision of $W_{k \times k}$ or the
line graph of a subdivision of $W_{k \times k}$ as an induced subgraph. 
\end{conjecture}
This is still open, while several interesting special cases of it are proved in earlier papers of this series \cite{wallpaper, ACV}. In the same vein, the following may be true as far as we know (this is a variant of a conjecture of \cite{mainconj}): 
\begin{conjecture}
\label{wallconjnobdd}
For all $t \geq 0$ there exists $c=c(t)$ such that if $G$ is a graph with no $K_t$, no $K_{t, t}$, no subdivision of $W_{t \times t}$ and no line graph of a subdivision of $W_{t \times t}$ as induced subgraphs, then $\tw(G) \leq c\log(|V(G)|)$. 
\end{conjecture}
We conclude this section with the following result, which is an immediate consequence of the Helly property of subtrees of a tree:

\begin{theorem}[\cite{diestel}]
  \label{cliqueinbag}
  Let $G$ be a graph, let $K$ be a clique of $G$, and let $(T, \chi)$ be a tree decomposition of $G$. Then, there is $v \in V(T)$ such that $K \subseteq \chi(v)$.
  \end{theorem}

\subsection{Proof outline and organization}

Let us now discuss the main ideas of the proof of
Theorem~\ref{mainthm}. We will give precise definitions of the concepts used below later in the paper; our goal here is to sketch  a road map of where
we are going. By Theorem~\ref{nocliquecutset} we may assume that the graph
in question does not admit a clique cutset.
By Theorem~\ref{contain_cube} we may
restrict our attention to cube-free graphs (the ``cube'' is a certain eight vertex graph defined later). Obtaining a tree decomposition is usually closely related to
producing  a collection of ``non-crossing decompositions,'' which  roughly
means that the decompositions
``cooperate'' with each other, and the pieces that are obtained when the graph
is simultaneously decomposed by all the decompositions in the collection
``line up'' to form a tree structure.

In the case of graphs in $\mathcal{C}_t$, there is a natural family of decompositions to turn to;  they correspond to special vertices of the graph called
``hubs,'' and are discussed in Section~\ref{cutsets}.
Unfortunately,
these natural decompositions are very far from being non-crossing, and therefore
we cannot use them in traditional ways to get tree decompositions. We were able
to overcome this issue in \cite{ACV} by using a bound on the maximum degree of the graph, but the same methods do not apply when no such bound exists.
What we can do instead is use degeneracy (that is given by  a result of
\cite{KuhnOsthus}) to partition the set
of all hubs (which yields a partition of all the natural decompositions)
of an $n$-vertex  graph $G$ in $\mathcal{C}_t$  into collections $S_1, \dots, S_p$, where each $S_i$ is ``non-crossing''
(this property is captured in Lemma~\ref{looselylaminar}), $p \leq C(t) \log n$ (where $C(t)$ only depends on $t$ and works for all $G \in \mathcal{C}_t$),
and vertices in $S_i$ have a bounded (as a function of $t$) number of
neighbors in
$\bigcup_{j=i}^p S_j$. Our main result is that the treewidth of $G$  is
bounded by a  linear  function of $p+\log n$.

It follows immediately
  from  Theorem~\ref{structured} and Theorem~\ref{smallPMC} that the treewidth of a
  graph in $\mathcal{C}_t$ is bounded (by a constant depending on $t$) if
  $G$ has no hubs; thus we may assume that $p>0$.
    It is sometimes the case that one of the hubs we get is not ``useful'' to us,
  and then we turn to the arguments in
  Section~\ref{sec:neighborhood_construction}.

In the general case, for  $p>0$, we proceed as follows.   We first decompose
  $G$, simultaneously, by all the decompositions corresponding to the hubs in $S_1$.
  This allows us to define a natural induced subgraph $\beta(S_1)$ of $G$
  that we call the ``central bag'' for $S_1$. The parameter $p$
  is smaller for $\beta(S_1)$ than it is for $G$, and so we can use induction
  to obtain a bound on the treewidth of $\beta(S_1)$.
  We then  start with a special optimal tree decompositon of $\beta(S_1)$,
  where each bag is a ``potential maximal clique'' (see Section~\ref{PMC}).
  Also inductively (this time on the number of vertices) we have tree
  decompositions for each component of $G \setminus \beta(S_1)$.

  Now  we use the special nature of our ``natural decompositions'' and properties of potential maximal cliques to combine the
  tree decompositions above into a tree decomposition of $G$, where the size
  of the bag only grows by an additive constant. This is possible
  because in the growing process   all we need to do is
  add to each existing bag the neighbor sets of some non-hub vertices in that
bag.    As the number of such vertices in each bag is bounded
  by Theorem~\ref{smallPMC}, and  due to the ``degeneracy'' propery of
  the partition $S_1, \dots, S_p$, we can ensure a bound on the growth.

  The paper is organized as follows. In Section~\ref{PMC} we introduce potential maximal cliques and discuss their properties. In Section~\ref{cutsets} we
  prove structural results guaranteeing the existence of useful decompositions.
  In Section~\ref{Hubs} we discuss the bounds on the number of non-hub vertices
  in minimal separators and in potential maximal cliques. Section~\ref{sec:neighborhood_construction} contains lemmas that will allow us to deal with hubs
  that do not fall into the framework of Section~\ref{sec:centralbag}.
  In Section~\ref{sec:centralbag} we discuss collections of non-crossing
  decompositions and properties of their central bags.  In Section~\ref{sec:collections} we show how to construct the partition $S_1, \dots, S_p$.
  Section~\ref{sec:proof} puts together the results of all the previous sections
  to prove Theorem~\ref{mainthm}. Finally, Section~\ref{algsec}
  discusses algorithmic consequences of Theorem~\ref{mainthm}.

\section{Potential maximal cliques} \label{PMC}

In the proof of Theorem~\ref{mainthm} we will use a special kind of tree
decomposition that we explain now.

For a graph $G$ and a set $F \subseteq \binom{V(G)}{2} \setminus E(G)$, we denote by $G+F$ the graph obtained from $G$ by making the pairs in
$F$ adjacent. 
A set $F \subseteq \binom{V(G)}{2} \setminus E(G)$ 
is a \emph{chordal completion}
or \emph{fill-in} of $G$ if $G+F$ is chordal; a chordal completion is \emph{minimal}
if it is inclusion-wise minimal. 
Let $X \subseteq V(G)$. The set $X$ is a \emph{minimal separator} if there exist $u, v \in V(G)$ such that $u$ and $v$ are in different connected components of $G \setminus X$, and $u$ and $v$ are in the same connected component of $G \setminus Y$ for every $Y \subsetneq X$.
A component $D$ of $G \setminus X$  is a \emph{full component} for $X$ if $N(D) = X$. It is well-known that a set $X \subseteq V(G)$ is a minimal separator if and only if there are at least two distinct full components for $X$.

A \emph{potential maximal clique} (PMC) of a graph $G$ is a set $\Omega \subseteq V(G)$ such that $\Omega$ is a maximal clique of $G + F$ for some minimal chordal completion $F$ of $G$. The following result
of \cite{BouchitteT01} characterizes PMCs: 

\begin{theorem}
\label{thm:PMC_characterization}
A set $\Omega \subseteq V(G)$ is a PMC of $G$ if and only if: 
\begin{enumerate}
    \item \label{PMC1} for every distinct $x, y \in \Omega$ with $xy \not \in E(G)$, there exists a component $D$ of  $G \setminus \Omega$ such that $x, y \in N(D)$. 
    \item for every component $D$ of $G \setminus \Omega$ it holds that $N(D) \subsetneq \Omega$.
\end{enumerate}
\end{theorem}
If $\Omega \subseteq V(G)$ and $D$ is a component of $G \setminus \Omega$ with $x, y \in N(\Omega)$ non-adjacent (as in \eqref{PMC1} above), we say that $D$ \emph{covers} the non-edge $xy$. 

We also need the following result  of \cite{BouchitteT01} relating PMCs and
minimal separators: 

\begin{theorem}
\label{prop:PMC_adhesions_are_seps}
Let $\Omega \subseteq V(G)$ be a PMC of $G$. Then, for every component $D$ of  $(G \setminus \Omega)$, the set $N(D)$ is a minimal separator of $G$. 
\end{theorem}

Let us say that a tree decomposition $(T, \chi)$ of a graph $G$ is {\em structured} if $\chi(v)$ is a PMC of $G$ for every $v \in V(T)$.
We denote by $\omega(G)$ the maximum size of a clique in $G$.
The following is a striking but easy fact (this was observed by multiple authors in the past, but we include the proof here for completeness):

\begin{theorem}
  \label{structured}
  Every graph $G$ has a structured tree decomposition of width $\tw(G)$.
  \end{theorem}

\begin{proof}
  Let $(T',\chi')$ be a tree decomposition of $G$ of width $\tw(G)$. It is easy to check that
  the graph $G'$ obtained from $G$ by adding all edges $xy$ such
  that $x,y \in \chi(v)$ for some $v \in V(T)$ is chordal. It follows that
  there exists a minimal chordal completion $F$ of $G$ such that
  $F \subseteq E(G') \setminus E(G)$; let $G''=G+F$.  In particular, every clique of $G''$ is a subset of a clique of $G'$.
  Since by Theorem~\ref{cliqueinbag} every clique of $G'$ is contained in a bag $\chi(v)$ for some $v \in V(T)$, it follows that $\omega(G'') \leq \omega(G') \leq \tw(G)+1$.
  Next, since $G''$ is chordal, there is a tree decomposition $(T'',\chi'')$ of $G''$  such that $\chi''(v)$ is a clique of $G''$ (and therefore a PMC of $G$)
  for every $v \in V(T'')$. 
  Lastly, since $G$ is a subgraph of $G''$, it follows that $(T'',\chi'')$
  is a tree decomposition of $G$. Since $\omega(G'') \leq \tw(G)+1$,
  it follows that $\width(T'', \chi'')=\tw(G)$, as required.
  This proves Theorem~\ref{structured}.
\end{proof}

\section{Structural results} \label{cutsets}
In this section we establish some useful structural properties of (theta, pyramid, generalized prism)-free graphs containing either a `cube' or a `wheel'. Let us define these notions and state our theorems properly. 

The \textit{cube} is the graph with vertex set $\{a_1,\ldots,a_6,b_1,b_2\}$ in which $\{a_1,\ldots,a_6\}$ is a hole, $b_1$ is complete to $\{a_1,a_3,a_5\}$, $b_2$ is complete to $\{a_2,a_4,a_6\}$, and there are no other edges. Let $G$ be a graph. We say a graph $H$ is a \textit{clique blow-up} of $G$ if $V(H)$ is the disjoint union of $|V(G)|$ non-empty and pairwise disjoint cliques $(X_v;v\in G)$ such that for all distinct $u,v\in G$, if $uv\in E(G)$ then $X_u$ is complete to $X_v$ in $H$, and if $uv\notin E(G)$ then $X_u$ is anticomplete to $X_v$ in $H$. A partition $(V_1,V_2)$ of the vertex set of a graph $G$ is said to be a \textit{cube partition} if $V_1$ is a clique blow-up of the cube, $V_2$ is a clique and $V_1$ is complete to $V_2$. 

According to our first result, the following, it turns out that even with only thetas and pyramids excluded, containing a cube results in a structurally simple class of graphs.

\begin{theorem}\label{contain_cube}
Let $G$ be a (theta, pyramid)-free graph. If $G$ contains a cube, then $G$ admits either a clique cutset or a cube partition.
\end{theorem}

We continue with more definitions. Let $G$ be a graph. Let $W$ be a hole in $G$ and $v \in G \setminus W$. A {\em sector} of $(W,v)$ is a path $P$ of $W$ of length at least one, such that both ends of $P$ are  adjacent to $v$ and $v$ is
anticomplete to $P^*$. A sector $P$ is {\em long} if $P^* \neq \emptyset$. 
A {\em wheel} in $G$ is a pair $(W,v)$ where $W$ is a hole of length at least five, $v$ has at least three neighbors in $W$ and $(W,v)$ has at least two long sectors.
For $v \in V(G)$ a wheel
$(W,v)$ is {\em optimal} if for every  wheel $(W',v)$  in $G$ we have
$|N_W(v)| \leq |N_{W'}(v)|$.
Let $\mathcal{C}^*$ be the class of all the cube-free graphs in $\mathcal{C}$.
Our second result is the following.

\begin{restatable}{theorem}{starcutset}
  \label{starcutset}
  Let $G \in \mathcal{C}^*$ and let $(W,v)$ be an optimal wheel in $G$.
  Then there is no component  $D$ of $G \setminus N[v]$
such that $W \subseteq N[D]$. 
\end{restatable}

Theorem~\ref{starcutset} is used later to prove Theorem~\ref{nohub}, which is a key step in our proof. It may aslo be of independent interest in the study of graphs in $\mathcal{C}$.
We prove Theorems~\ref{contain_cube} and~\ref{starcutset} in the upcoming two subsections. Let us conclude with a theorem concerning degeneracy. Recall that for an integer $\delta > 0$, a  graph $G$ is {\em $\delta$-degenerate} if every
  subgraph of $G$ contains a vertex of degree less than $\delta$. The following is an easy consequence of the main theorem of \cite{KuhnOsthus}:
\begin{theorem}\label{degenerate}
  For every $t\geq 1$, there exists $\delta_t > 0$ such that every (theta, $K_t$)-free graph is $\delta_t$-degenerate.
\end{theorem}

To deduce Theorem~\ref{degenerate} from \cite{KuhnOsthus} observe that
theta-free graphs do not contain subdivisions of the complete bipartite graph
$K_{2,3}$ as induced subgraphs.

\subsection{Cube attachments}
Here we prove Theorem~\ref{contain_cube}.
\begin{proof}[Proof of Theorem~\ref{contain_cube}]
Suppose not. Let $S$ be the largest subset of $V(G)$ admitting a cube partition. Since $G$ contains a cube, say $Q$, it follows that $V(Q)$ admits a cube partition with $V_1 = V(Q)$ and $V_2 = \emptyset$. This shows that $S$ exists. Since $G$ does not admit a cube partition, it follows that $S \neq V(G)$. Let $(V_1, V_2)$ be a cube partition of $S$. We may assume that $V_1$ admits a partition into eight non-empty cliques $A_1,\ldots,A_6,B_1,B_2$, such that
\begin{itemize}
    \item for each $i\in \{1,\ldots,6\}$, $A_i$ is complete to $A_{i+1}$ (where $A_7=A_1$); and
     \item $B_1$ is complete to $A_1\cup A_3\cup A_5$ and $B_2$ is complete to $A_2\cup A_4\cup A_6$;
\end{itemize}
and there no more edges in $V_1$.

\sta{\label{cub1} Let $w\in G\setminus S$ have two non-adjacent neighbors $x,y\in S$. Then $x$ is at distance two from $y$ in $S$.}

Suppose not. By symmetry, we may assume that $x\in B_1$ and $y\in B_2$. Since $A_2\cup A_3\cup A_5\cup A_6\cup B_1\cup B_2\cup \{w\}$ contains no theta, $w$ is complete to at least one of $A_2,A_3,A_5,A_6$, say $A_6$. Also, since $A_1\cup A_5\cup A_6\cup B_1\cup \{w\}$ contains no theta, $w$ is complete to at least one of $A_1,A_5$; by symmetry, let $w$ be complete to $A_5$. Suppose that $w$ has a non-neighbor in one of $A_1$ and $A_4$, say the former. Since $A_1\cup A_2\cup A_6\cup B_1\cup \{w\}$ contains no theta, $w$ is anticomplete to $A_2$. As a result, $w$ is complete to $A_4$, as otherwise depending on whether $w$ has a neighbor in $A_3$ or not, either $A_2\cup A_3\cup A_4\cup B_2\cup \{w\}$ or $A_2\cup A_3\cup A_4\cup B_1\cup B_2\cup \{w\}$ contains a theta, which is impossible. Consequently, since $A_2\cup A_3\cup A_4\cup B_1\cup B_2\cup \{w\}$ contains no pyramid, $w$ is complete to $A_3$. But then $A_1\cup A_2\cup A_3\cup B_1\cup B_2\cup \{w\}$ contains a pyramid, a contradiction. This proves that  $w$ is complete to both $A_1$ and $A_4$. Next suppose that $w$ has a non-neighbor in one of $A_2$ or $A_3$, say the former. Since $A_1\cup A_2\cup A_3\cup B_2\cup \{w\}$ contains no theta, $w$ is anticomplete to $A_3$. But then $A_2\cup A_3\cup A_4\cup B_1\cup B_2\cup \{w\}$ contains a pyramid, a contradiction. Therefore, $w$ is complete to both $A_2$ and $A_3$. This restores the symmetry between $B_1$ and $B_2$.
Finally, if $w$ has a non-neighbor $u\in B_1\cup V_2$, then $A_1\cup A_3\cup A_5\cup \{u,w\}$ contains a theta. So $w$ is complete to both $B_1$ and $V_2$,
and symmetrically to $B_2$. In conclusion, $w$ is complete to $V_1\cup V_2$. But then $(V_1,V_2\cup \{w\})$ is a cube partition for $S\cup \{w\}$, a contradiction with the choice of $S$. This proves \eqref{cub1}.

\sta{\label{cub2} For every $w\in G\setminus S$, $N_S(w)$ is a clique.}

Suppose not. Then by \eqref{cub1} and symmetry, we may assume that $w$ has a neighbor in $A_1$ and a neighbor in $A_3$. It follows from \eqref{cub1} that $w$ is anticomplete to $A_4$ and $A_6$. Also, since $A_3\cup A_4\cup A_5\cup B_2\cup \{w\}$ contains no theta, $w$ is anticomplete to at least one of $A_5,B_2$; by symmetry, let $w$ be anticomplete to be $A_5$. Now, since $A_1\cup \cdots \cup A_6\cup \{w\}$ contains no theta, $w$ is complete to $A_2$. In addition, since $A_1\cup A_3 \cup A_4\cup A_5\cup A_6\cup B_2\cup \{w\}$ contains no theta, $w$ is complete to $B_2$. The latter, along with \eqref{cub1}, implies that $w$ is anticomplete to $B_1$. Moreover, if $w$ has a non-neighbor in one of $A_1$ and $A_3$, say the former, then $A_1\cup A_3 \cup A_4\cup A_6\cup B_1\cup B_2\cup \{w\}$ contains a theta, which is impossible. Thus, $w$ is complete to $A_1\cup A_3$. Finally, if $w$ has a non-neighbor $u\in V_2$, then $A_1\cup A_3\cup B_2\cup \{u,w\}$ contains a theta. So $w$ is complete to $V_2$. In conclusion, $w$ is complete to $A_1,A_2,A_3,B_2$ and $V_2$ and anticomplete to the rest of $S$. But then adding $w$ to $A_2$, $(V_1\cup \{w\},V_2)$ is a cube partition for $S\cup \{w\}$, a contradiction with the choice of $S$. This proves \eqref{cub2}.

\sta{\label{cub3} For every component $J$ of $G\setminus S$, $N_S(J)$ is a clique.}

Suppose not. Then we may choose an induced path $R=r_1\mh\cdots \mh r_l$ of smallest length such that $R\subseteq G\setminus S$, and (having chosen $R$) non-adjacent vertices $x\in N_S(r_1)$ and $y\in N_S(r_l)$ such that the distance between $x$ and $y$ in $S$ is as small as possible. It follows from \eqref{cub2} that $l\geq 2$. Now, if the distance between $x$ and $y$ in $S$ is three, we may assume, without loss of generality, that $x\in B_1$ and $y\in B_2$. From the choice of $R$, $x$ and $y$, and the fact that every vertex in $V_1\setminus (B_1\cup B_2)$ is at distance two from either $x$ or $y$, it follows that  that $R^*$ is anticomplete to $V_1$, $N_S(r_1)\subseteq B_1$ and $N_S(r_l)\subseteq B_2$. But then $A_2\cup A_3\cup A_5\cup A_6\cup R\cup \{x,y\}$ contains a theta, which is impossible. Therefore, the distance between $x$ and $y$ in $S$ is two, and so due to symmetry, we may assume that $x\in A_1$ and $y\in A_3$. From the choice of $R$, it immediately follows that $R^*$ is anticomplete to $A_4\cup A_5\cup A_6\cup B_2$, $r_1$ is anticomplete to $A_4\cup A_5\cup B_2$ and $r_l$ is anticomplete to $A_5\cup A_6\cup B_2$. But then $A_1\cup A_3\cup A_4\cup A_5\cup A_6\cup B_2\cup R$ contains a theta, which is impossible. This proves \eqref{cub3}.\vsp

Since $S\neq V(G)$, $G\setminus S$ has a component $J$. By \eqref{cub3}, $N_S(J)$ is a clique. So $S\setminus N(J)\neq \emptyset$, as $S$ is not a clique. But then $N_S(J)$ is a cutset in $G$ separating $J$ from $S\setminus N(J)$, a contradiction. This concludes the proof of Theorem~\ref{contain_cube}.
\end{proof}

\subsection{Wheel attachments}
The goal of this subsection is to prove Theorem~\ref{starcutset}, which falls into several steps. First, we need a couple of definitions. Let $G$ be a graph and $(W,v)$ be a wheel in $G$. Throughout, we fix a cyclic orientation of $W$ and refer to it as \textit{clockwise}. For all $x,y\in W$, we denote by $W[x,y]$ the subpath of $W$ joining $x$ to $y$ in the clockwise orientation, with the convention that $W[x,x]=\{x\}$. Also, for every $z\in W$, let $z^-\in W$ (resp. $z^+\in W$) be the vertex which appears immediately before (resp. after) $z$ with respect to the clockwise orientation of $W$. For every vertex $w\in G\setminus (N[v]\cup W)$, we say $w$ is $(W,v)$-\textit{local} if $N_W(w)$ is contained in a sector of $(W,v)$. Similarly, a component $J$ of $G\setminus N[v]$ is said to be $(W,v)$-\textit{local} if $N_W[J]$ is contained in a sector of $(W,v)$ (note that a local component may meet $W$).

A wheel $(W,v)$ in a graph $G$ is \textit{stranded} if for some $k\geq 2$, there exists a (unique) enumeration $a_1,\ldots, a_k,b$ of the vertices in $N_W(v)$ such that the following hold.
\begin{itemize}
    \item The vertices $a_1,\ldots, a_k,b$ appear in this order with respect to the clockwise orientation of $W$.
    \item For each $i\in \{1,\ldots, k-1\}$, we have $a_{i+1}=a_i^+$.
    \item Both $W[b,a_1]$ and $W[a_k,b]$ are long sectors of $(W,v)$.
\end{itemize}
We refer to $(a_1,\ldots, a_k,b)$ as the \textit{contour} of $(W,v)$. Note that if $k=2$, then $W\cup \{v\}$ is a pyramid in $G$. Let us begin with a lemma.

\begin{lemma}\label{strand-vx}
Let $G$ be a (theta, pyramid, prism)-free graph and $(W,v)$ be an optimal wheel in $G$. Assume that $(W,v)$ is a stranded wheel  with contour $(a_1,\ldots, a_k,b)$. Then for every $w\in G\setminus (N[v]\cup W)$, either $w$ is $(W,v)$-local or $N_W(w)=\{b^-,b,b^+\}$.
\end{lemma}

\begin{proof}
Suppose for a contradiction that $w$ is not $(W,v)$-local and $N_W(w)\neq \{b^-,b,b^+\}$. In this proof, by $N(w)$ we mean $N_W(w)$. Note that since $G$ has no pyramid, we have $k\geq 3$.

\sta{\label{str1} The vertex $w$ has a neighbor in $W\setminus \{a_1,\ldots, a_k\}$.}

Suppose not. Then since $w$ is not $(W,v)$-local, it has a least two neighbors in $\{a_1,\ldots, a_k\}$. Choose $i\in \{1,\ldots, k\}$ minimum such that $a_i\in N(w)$, and $j\in \{1,\ldots, k\}$ maximum such that $a_j\in N(w)$. Again, since $w$ is not $(W,v)$-local, we have $j\geq i+2$. But then $W'=(W\setminus W[a_{i+1},a_{j-1}])\cup \{w\}$ is a hole in $G$, and $(W',v)$ is a wheel in $G$ with $|N_{W'}(v)|< |N_W(v)|$, a contradiction with the optimality of $(W,v)$. This proves \eqref{str1}.

\sta{\label{str2} Neither $N(w)\cap W[b^+,a_1^-]$ nor $N(w)\cap W[a_k^+,b^-]$ is empty.}

For otherwise by symmetry we may assume that $N(w)\cap W[b^+,a_1^-]=\emptyset$. If $w$ has no neighbor in $\{a_1,\ldots, a_{k-1}\}$, then $w$ is $(W,v)$-local, a contradiction. So we may choose $i\in \{1,\ldots, k-1\}$ minimum such that $a_i\in N(w)$. Also, by \eqref{str1}, $w$ has a neighbor in $W[a_k^+,b]$. Traversing the path $W[a_k^+,b]$ from $a_k^+$ to $b$, let $z$ be the last vertex in $N(w)\cap W[a_k^+,b]$. Now, if $i=1$ (resp. $i=2$), then $W[z,a_i]\cup \{v,w\}$ is a theta (resp. pyramid) in $G$, which is impossible. Therefore, we have $i\geq 3$, and so $k\geq 4$. But then $W'=(W\setminus W[a_{i+1},z^-])\cup \{w\}$ is a hole in $G$, and $(W',v)$ is a wheel in $G$ with $|N(v)\cap W'|< |N(v)\cap W|$, a contradiction with the optimality of $(W,v)$. This proves \eqref{str2}.

\sta{\label{str3} We have $b\in N(w)$.}

Suppose not. By \eqref{str2}, neither $N(w)\cap W[b^+,a_1^-]$ nor $N(w)\cap W[a_k^+,b^-]$ is empty. Traversing the path $W[b^+,a_1^-]$ from $b$ to $a_1$, let $x$ be the first vertex in $N(w)\cap W[b^+,a_1^-]$. Also, traversing the path $W[a_k^+,b^-]$ from $a_k^+$ to $b^-$, let $z$ be the last vertex in $N(w)\cap W[a_k^+,b]$. If $N(w)\cap W=\{x,z\}$, then $W\cup \{w\}$ is a theta in $G$, which is impossible. Also, if $w$ is adjacent to $a_i$ for some $i\in \{2,\ldots, k-1\}$, then $W[z,x]\cup \{a_i,v,w\}$ is a theta in $G$, a contradiction. Therefore, $w$ has a neighbor in either $W[a_k,z^-]$ or $W[x^+,a_1]$, say the former. Traversing the path $W[a_k,z^-]$ from $a_k$ to $z^-$, let $y$ be the first vertex in $N(w)\cap W[a_k,z^-]$. Then depending on whether $y$ is adjacent to $z$ or not, $W[z,x]\cup W[a_k,y]\cup \{v,w\}$ is either a pyramid or a theta in $G$, a contradiction. This proves \eqref{str3}.\vsp

By \eqref{str2}, neither $N(w)\cap W[b^+,a_1^-]$ nor $N(w)\cap W[a_k^+,b^-]$ is empty. Traversing the path $W[b^+,a_1]$ from $b^+$ to $a_1$, let $x$ be the last vertex in $N(w)\cap W[b^+,a_1]$.  Also, traversing the path $W[a_k,b^-]$ from $a_k$ to $b^-$, let $z$ be the first vertex in $N(w)\cap W[a_k,b^-]$. Due to symmetry, we may assume that $|W[b,x]|\geq |W[z,b]|$. If $|W[b,x]|\geq 3$, then depending on whether $|W[z,b]|\geq 3$ or not,  $W[x,a_1]\cup W[a_k,z]\cup \{b,v,w\}$ is either a theta or a pyramid in $G$, a contradiction. We deduce that $|W[b,x]|=|W[z,b]|=2$, from which we have $x=b^+$, $z=b^-$ and $N(w)\cap W[a_k^+,a_1^-]=\{b^-,b,b^+\}$. Furthermore, by the assumption, we have $N(w)\cap W\neq \{b^-,b,b^+\}$. So $w$ has a neighbor in $\{a_1,\ldots, a_k\}$. Now, if $w$ is adjacent to both $a_1$ and $a_k$, then $\{a_1,a_k,b,v,w\}$ is a theta in $G$, which is impossible. Therefore, we may assume, without loss of generality, that $w$ is not adjacent to $a_1$,
and so $a_i\in N(w)$ for some $i\in \{2,\ldots, k\}$. But then depending on whether $i=2$ or not, $W[b,a_1]\cup \{a_i,v,w,x\}$ is either a prism or a pyramid in $G$, a contradiction. This completes the proof of Lemma~\ref{strand-vx}.
\end{proof}

From Lemma~\ref{strand-vx}, we deduce the following.

\begin{lemma}\label{strand-general}
Let $G$ be a (theta, pyramid, prism)-free graph and $(W,v)$ be an optimal wheel in $G$. Assume that $(W,v)$ is a stranded wheel  with contour $(a_1,\ldots, a_k,b)$. Then no component of $G\setminus (N[v]\setminus \{a_2,\ldots, a_{k-1}\})$ contains $W\setminus \{a_1,a_k, b\}$.
\end{lemma}
\begin{proof}
Suppose not. Note that since $G$ has no pyramid, we have $k\geq 3$. Then we may choose an induced path $P=p_1\mh\cdots\mh p_l$ in $G$ such that the following hold.

\begin{enumerate}[(P1)]
		\item\otherlabel{p1}{P1} We have $P\subseteq G\setminus (W\cup N[v])$.
    \item\otherlabel{p2}{P2} The vertex $p_1$ has a neighbor in $\{a_2,\ldots, a_{k-1}\}$ and $p_l$ has a neighbor in $W[a_k^+,a_1^-]\setminus \{b\}$.
    \item\otherlabel{p3}{P3} Subject to \eqref{p1} and \eqref{p2}, $|P|$ is as small as possible.
\end{enumerate}

It follows from Lemma~\ref{strand-vx} that $l\geq 2$, $N_W(p_1)\subseteq \{a_1,\ldots,a_k\}$ and either $p_l$ is $(W,v)$-local or $N_W(p_l)=\{b^-,b,b^+\}$. Also, we deduce:

\sta{\label{str4} $P^*$ is anticomplete to $(W\cup \{v\})\setminus \{a_1,a_k\}$.}

Suppose for a contradiction that $p_i$ has a neighbor $q\in (W\cup \{v\})\setminus \{a_1,a_k\}$ for  some $i\in \{2,\ldots, l-1\}$. Since $P \cap N[v]=\emptyset$, it follows that $q \neq v$.  Then depending on whether $q\in \{a_2,\ldots, a_{k-1}\}$ or $q\in W[a_k^+,a_1^-]$, either $p_i\mh P\mh p_l$ or $p_1\mh P\mh p_{i}$ is an induced path in $G$ with fewer vertices than $P$ and satisfying \eqref{p1} and \eqref{p2}, which violates \eqref{p3}. This proves \eqref{str4}.

\sta{\label{str5} We have $N_W(p_l)=\{b^-,b,b^+\}$.}

Suppose not. Then $p_l$ is $(W,v)$-local, and we may assume, without loss of generality, that $N_W(p_l)\subseteq W[a_k,b]$. Note that by \eqref{p2}, $p_l$ has a neighbor in $W[a_k^+,b]$. So  traversing the path $W[a_k^+,b]$ from $a_k^+$ to $b$, let $z$ be the last vertex in $N(p_l)\cap W[a_k^+,b]$. Also, note that again by \eqref{p2}, $p_1$ has a neighbor in $\{a_2,\ldots, a_{k-1}\}$, while $p_l$ is anticomplete to $\{a_1,\ldots, a_{k-1}\}$. So let $i\in \{1,\ldots, l-1\}$ be maximum such that $p_i$ has a neighbor in $\{a_1,\ldots, a_{k-1}\}$, and choose $j\in \{1,\ldots, k-1\}$ minimum such that $p_i$ is adjacent to $a_j$. Now, if $j=1$ (resp. $j=2$), then $W[z,a_j]\cup \{p_i,\ldots, p_l\}\cup \{v\}$ is a theta (resp. pyramid) in $G$, which is impossible. Therefore, we have $j\geq 3$, and so $k\geq 4$. But then $W'=(W\setminus W[a_{j+1},z^-])\cup \{p_i,\ldots,  p_l\}$ is a hole in $G$, and $(W',v)$ is a wheel in $G$ with $|N_{W'}(v)|< |N_W(v)|$, a contradiction with the optimality of $(W,v)$. This proves \eqref{str5}.

\sta{\label{str6} The vertex $p_1$ is anticomplete to $\{a_1,a_k\}$ and $P^*$ is anticomplete to $W\cup \{v\}$.}

Suppose not. Then by \eqref{str4}, some vertex in $P\setminus \{p_l\}$ has a neighbor in $\{a_1,a_k\}$, while $P^*$ is anticomplete to $(W\cup \{v\})\setminus \{a_1,a_k\}$. Let us choose $i\in \{1,\ldots,l-1\}$ maximum such that $p_i$ is adjacent to either $a_1$ or $a_k$, say the latter. Also, by \eqref{str5}, we have $N(p_l)\cap W=\{b^-,b,b^+\}$. But then $W[a_k,b]\cup \{p_i,\ldots, p_l\}\cup \{a_k,b^-,b,v\}$ is a pyramid in $G$, which is impossible. This proves \eqref{str6}.\vsp

Now, choose $i\in \{2,\ldots,k-1\}$ minimum such that $p_1$ is adjacent to $a_i$. Then depending on whether $i=2$ or not, $P\cup W[b,a_1]\cup  \{a_i,v\}$ is either a prism or a pyramid in $G$, which is impossible. This concludes the proof of Lemma~\ref{strand-general}.
\end{proof}

\begin{lemma}\label{non-str-vx}
Let $G\in \mathcal{C}^*$ and $(W,v)$ be an optimal wheel in $G$ which is not stranded. Then for every $w\in G\setminus (N[v]\cup W)$, $w$ is $(W,v)$-local.
\end{lemma}
\begin{proof}
Suppose not. We deduce:

\sta{\label{nstr-vx1}We have $|N_W(w)|\geq 3$.}

For otherwise either $w$ is $(W,v)$-local or $W\cup \{w\}$ is a theta in $G$. This proves \eqref{nstr-vx1}.

\sta{\label{nstr-vx2} The vertex $w$ has a neighbor in $W\setminus N(v)$.}

Suppose not. If $N_W(w)$ contains a stable set $S$ of size three, then $S\cup \{v,w\}$ is a theta in $G$, which is impossible. So $N_W(w)$ has no stable set of size three. Now, if $N_W(w)$ is not connected, then by \eqref{nstr-vx1}, $N_W(w)$ is either an edge plus an isolated vertex, or an induced two-edge matching. In the former case, $W\cup \{w\}$ is a pyramid, and in the latter case, $W\cup \{w\}$ is a pinched prism, violating the fact that $G\in \mathcal{C^*}$. So $N_W(w)$ is connected, and so since $w$ is not $(W,v)$-local, $N_W(w)$ is a path on three or four vertices, say $w_1\mh \cdots \mh w_l$ for $l\in \{3,4\}$. But then $W'=(W\setminus \{w_2,\ldots, w_{l-1}\})\cup \{w\}$ is hole in $G$, and $(W',v)$ is a wheel in $G$ with $|N_{W'}(w)|< |N_W(w)|$, a contradiction with the optimality of $(W,v)$. This proves \eqref{nstr-vx2}.

\sta{\label{nstr-vx3}There do not exist two consecutive sectors of $(W,v)$ whose union contains $N_W(w)$.}

For otherwise there are two consecutive sectors of $(W,v)$, say $W[a,b]$ and $W[b,c]$, where $N_W(w)\subseteq W[a,c]$. Traversing $W[a,c]$ from $a$ to $c$, let $x$ be first vertex in $N(w)\cap W[a,c]$ and let $z$ be last vertex in $N(w)\cap W[a,c]$. Note that since $w$ is not $(W,v)$-local, we have $x\in W[a,b^-]$ and $z\in W[b^+,c]$. Also, since $(W,v)$ is not stranded, $W[c,a]$ contains a long sector of $(W,v)$. But then $W'=(W\setminus W[x^+,z^-])\cup \{w\}$ is a hole in $G$, and $(W',v)$ is wheel in $G$ with $|N_{W'}(v)|< |N_W(v)|$, a contradiction with the optimality of $(W,v)$. This proves \eqref{nstr-vx3}.

\sta{\label{nstr-vx4} Let $W[r,s]$ be a long sector of $(W,v)$ such that $w$ has a neighbor in the interior of $W[r,s]$. Then $w$ has no neighbor in $W\setminus W[r^-,s^+]$.}

Suppose not. Note that $v$ has a neighbor in $W\setminus W[r^-,s^+]$, as otherwise $W\cup \{v\}$ is a pyramid, a theta  or a pinched prism, which violates the fact that $G\in \mathcal{C}^*$. Therefore, $(W\setminus W[r^-,s^+])\cup \{v,w\}$ is connected, and we may choose an induced path $L$ from $v$ to $w$ contained in $(W\setminus W[r^-,s^+])\cup \{v,w\}$. Also, traversing $W[r,s]$ from $r$ to $s$, let $y$ be the first vertex in $N(w)\cap W[r,s]$ and $z$ be the last vertex in $N(w)\cap W[r,s]$. Note that since $w$ has a neighbor in $W[r^+,s^-]$, either $y$
 and $z$ are distinct or $y=z\in W[r^+,s^-]$.  Now, depending on whether $y$ is adjacent to $z$ or not, $W[r,y]\cup W[z,s]\cup L$ is a either pyramid or a theta in $G$, which is impossible. This proves \eqref{nstr-vx4}.\vsp 

Note that \eqref{nstr-vx2}, \eqref{nstr-vx2}\eqref{nstr-vx3} and \eqref{nstr-vx4} together with the fact that $(W,v)$ is a wheel immediately imply the following.

\sta{\label{nstr-vx5} Let $W[r,s]$ be a long sector of $(W,v)$ such that $w$ has a neighbor in the interior of $W[r,s]$. Then $r^-$ and $s^+$ are distinct, $w$ is adjacent to both $r^-$ and $s^+$, and $w$ has no neighbor in $W\setminus W[r^-,s^+]$.}

Now, by \eqref{nstr-vx2}, we may choose $a,b\in W$ such that $W[a,b]$ is a long sector of $(W,v)$ and $w$ has a neighbor in $W[a^+,b^-]$. Then by \eqref{nstr-vx5} applied to the sector $W[a,b]$, $a^-$ and $b^+$ are distinct, $w$ is adjacent to both $a^-$ and $b^+$, and $w$ has no neighbor in $W\setminus W[a^-,b^+]$.
Next we prove:

\sta{\label{nstr-vx6} $v$ is complete to $\{a^-,b^+\}$.}

  Suppose that $v$ is non-adjacent to $a^-$. Let $c \in W$ be such that
  $W[c,a]$ is a sector of $(W,v)$. Then $c \neq b$, for otherwise we get a theta. By \eqref{nstr-vx5}
  applied to $W[c,a]$
  we deduce that $w$ is complete to $\{a^+,c^-\}$ and
  has no other neighbor in $W[a^+,c^-]$. Since $w$ is adjacent to $b^+$,
  it follows that $c^-=b^+$, and $W[b,c]$ is a sector of $(W,v)$.
  Now it follows by symmetry that $a^+=b^-$ and $c^+=b^-$, and that
  $N_W(v) = \{a,b,c\}$ and $N_W(w) = W \setminus N(v)$.
  But now $W\cup \{v,w\}$ is a cube in $G$, which is impossible. This proves \eqref{nstr-vx6}.\vsp

  Let $W'=(W \setminus W[a,b])\cup \{w\}$. Then $W'$ is a hole and
  $|N_{W'}(v)| < |N_{W}(v)|$. By the optimality of $(W,v)$ we deduce that
  the $N_{W'}(v)$ forms a path in  $W'$. Since by \eqref{nstr-vx6}
  $v$ is complete to $\{a^-,b^+\}$, it follows that $v$ is complete to
  $W \setminus W[a,b]$. But then $N_W(v)$ forms a  path in $W$, contrary
  to the fact that $(W,v)$ is a wheel.
  This completes the proof of Lemma~\ref{non-str-vx}.
\end{proof}

\begin{lemma}\label{nstr-general}
Let $G\in \mathcal{C}^*$ and $(W,v)$ be an optimal wheel in $G$ which is not stranded. Then every component of $G\setminus N[v]$ is $(W,v)$-local.
\end{lemma}
\begin{proof}
Suppose for a contradiction that some component $J$ of $G\setminus N[v]$ is not $(W,v)$-local. In other words, there are two distinct vertices $u_1,u_2\in N_W(v)$ such that neither $N[J]\cap W[u_1^+,u_2^-]$ nor $N[J]\cap W[u_2^+,u_1^-]$ is empty. From this and Lemma~\ref{non-str-vx}, it is easily observed that, for fixed $v$, the hole $W$ in the non-stranded wheel $(W,v)$ together with an induced path $Q=q_1\mh\cdots\mh q_l$ in $G$ can be chosen such that the following hold.
\begin{enumerate}[(Q1)]
\item\otherlabel{q1}{Q1}We have $l\geq 2$ and $Q\subseteq G\setminus (W\cup N[v])$.
		\item\otherlabel{q2}{Q2} Both $q_1$ and $q_l$ have neighbors in $W$, and there are two distinct sectors $S$ and $S'$ of $(W,v)$ such that $N_W(q_1)\subseteq S$ and $N_W(q_l)\subseteq S'$. Also, no sector of $(W,v)$ contains $N_W(q_1)\cup N_W(q_l)$.
			\item\otherlabel{q3}{Q3} Subject to \eqref{q1} and \eqref{q2}, $|W|+|Q|$ is as small as possible.
		\item\otherlabel{q4}{Q4} Subject to \eqref{q1}, \eqref{q2} and \eqref{q3}, a component of  $W\setminus (S^*\cup S'^*)$ with the smallest number of vertices has as few vertices as possible.
\end{enumerate}
We deduce:

\sta{\label{nstr-g1} If $S\cap S'\neq \emptyset$, then some vertex in $Q^*$ has a neighbor in $W\setminus (S\cup S')$.}

Suppose not. Then $Q$ is anticomplete to $W\setminus (S\cup S')$. Traversing $S\cup S'$ in the clockwise orientation of $W$, let $y$ be the first vertex with a neighbor in $Q$ and $z$ be the last vertex with a neighbor in $Q$. Note that by \eqref{q2}, both $y$ and $z$ exist, one of them belongs to  $S\setminus S'$ and the other one lies in $S'\setminus S$. Then $Q\cup \{y,z\}$ is connected, and hence contains an induced path $R$ from $y$ to $z$. Also, since $(W,v)$ is not stranded, $W\setminus (S\cup S')^*$ contains a long sector of $(W,v)$. Now if $N(v)\setminus (S\cup S')=\emptyset$, then $W[z,y]\cup R \cup  \{v\}$ is a theta in $G$, which is impossible. So $v$ has a neighbor in $W\setminus (S\cup S')$. But then $W'=(W\setminus W[y,z])\cup R$ is a hole in $G$, and $(W',v)$ is a wheel in $G$ with $|N_{W'}(v)|< |N_W(v)|$, a contradiction with the optimality of $(W,v)$. This proves \eqref{nstr-g1}.

\sta{\label{nstr-g2} We have $S\cap S'=\emptyset$.}

Suppose not. Assume, without loss of generality, that $S=W[s,c]$ and $S'=W[c,s']$ for some $s,c,s'\in W$, appearing in this order with respect to the clockwise orientation of $W$. Then by \eqref{nstr-g1}, some vertex $q_i\in Q^*$ has a neighbor in $W\setminus (S\cup S')$. Now, if there exist $q_j\in Q$ with a neighbor in $(S\cup S')^*$ (in which case by Lemma~\ref{non-str-vx}, we have $i\neq j$), then assuming $Q'$ to be the subpath of $Q$ joining $q_i$ to $q_j$, $(W,v)$ and $Q'$ satisfy both \eqref{q1} and \eqref{q2}, yet $|Q'|<|Q|$, a contradiction with \eqref{q3}. So $Q$ is anticomplete to $(S\cup S')^*$. As a result, we have $N(q_1)\cap W=\{s\}$ and $N(q_l)\cap W=\{s'\}$.  It follows from \eqref{q2} that $W\setminus (S\cup S')^*$ is not a sector of $(W,v)$, and so $v$ has a neighbor in $W\setminus (S\cup S')$. Now, assume that there are two distinct vertices $x,y\in W\setminus (S\cup S')$ such that $x\in N_W(v)$ and $y\in N_W(Q^*)$, say $y\in W[s',x]$. Then defining $Q'$ to be the subpath of $Q$ from $q_1$ to $q_i$, $(W,v)$ and $Q'$ satisfy both \eqref{q1} and \eqref{q2}, yet $|Q'|<|Q|$, a contradiction with \eqref{q3}. Consequently, there exists a vertex $z\in W\setminus (S\cup S')$ such that $N_{W\setminus (S\cup S')}(v)=N_W(Q^*)=\{z\}$. Now, if $W[z,s]$ is a long sector of $(W,v)$, then assuming $R$ to be an induced path in $Q\cup \{s,z\}$ from $s$ to $z$, $R\cup W[z,s]\cup \{v\}$ is a theta in $G$, which is impossible. So we have $|W[z,s]|=2$, and similarly $|W[s',z]|=2$. But then $(W,v)$ is stranded, a contradiction. This proves \eqref{nstr-g2}.

\sta{\label{nstr-g3} Both $N_W(q_1)$ and $N_W(q_l)$ are cliques of $G$, and so $|N_W(q_1)|,|N_W(q_l)|\in \{1,2\}$.}

Suppose for a contradiction, and by symmetry that $q_1$ has two non-adjacent neighbors in $W$, which by \eqref{q2} belong to $S$. Therefore, traversing $S$ in the clockwise orientation of $W$, let $y$ and $z$ respectively be the first and the last vertex in  $N_S(q_1)=N_W(q_1)$. Note that $W[y^+,z^-]\neq\emptyset$. Now $W'=(W\setminus W[y^+,z^-])\cup \{q_1\}$ is a hole in $G$ and so $(W',v)$ is an optimal wheel in $G$, as $|N_{W'}(v)|=|N_W(v)|$. Also, $(W,v)$ and $(W',v)$ have the same number of long sectors. Thus, by Lemma~\ref{non-str-vx}, $q_l$ is $(W',v)$-local, and so $l\geq 3$. But then defining $Q'$ to be the subpath of $Q$ from $q_2$ to $q_l$, $(W',v)$ and $Q'$ satisfy both \eqref{q1} and \eqref{q2}, yet $|W'| \leq |W|$ and $|Q'|<|Q|$, a contradiction with \eqref{q3}. This proves \eqref{nstr-g3}.

\sta{\label{nstr-g4} $Q^*$ is anticomplete to $W$.} 

For otherwise there exists a vertex $q_i\in Q^*$ with a neighbor $x\in W$. By \eqref{nstr-g2}, either $x\notin S$ or $x\notin S'$, say the former holds. Let $S=W[a,b]$ and $S'=W[c,d]$, with distinct vertices $a,b,c,d$ appearing in this order with respect to the clockwise direction of $W$. If $x\in S'^*$, then defining $Q'$ to be the subpath of $Q$ from $q_1$ to $q_i$, $(W,v)$ and $Q'$ satisfy both \eqref{q1} and \eqref{q2}, yet $|Q'|<|Q|$, a contradiction with \eqref{q3}. So we may assume, by symmetry, that $x\in W[b^+,c]$. Also, if for some neighbor $y\in W[a,b]$ of $q_1$, $v$ has a neighbor in $W[y^+,x^-]$, then again defining $Q'$ to be the subpath of $Q$ from $q_1$ to $q_i$, $(W,v)$ and $Q'$ satisfy both \eqref{q1} and \eqref{q2}, yet $|Q'|<|Q|$, a contradiction with \eqref{q3}.  Therefore, we have $N_W(q_1)=\{b\}$ and $v$ has no neighbor in $W[b,x]\setminus \{b,x\}$. But then $x\neq c$, since otherwise we may replace $S$ by the sector $W[b,c]$ of $(W,v)$, which violates \eqref{q4}. Similarly, if $q_l$ has a neighbor in $z\in W[c,d]$ for which $v$ has a neighbor in $W[x^+,z^-]$, then defining $Q'$ to be the subpath of $Q$ from $q_i$ to $q_l$, $(W,v)$ and $Q'$ satisfy both \eqref{q1} and \eqref{q2}, yet $|Q'|<|Q|$, a contradiction with \eqref{q3}.  Consequently, we have $N_W(q_l)=\{c\}$ and $v$ has no neighbor in $W[x,c]\setminus \{x,c\}$. On the hand, by \eqref{q3}, no sector of $(W,v)$ contains $\{b,c\}$. Consequently, $v$ is adjacent to $x$. But then we can replace $S$ and $S'$ respectively with the sectors $W[b,x]$ and $W[x,c]$ of $(W,v)$, which contradicts \eqref{q4}. This proves \eqref{nstr-g4}.

Now, \eqref{nstr-g2} together with \eqref{nstr-g3}
and \eqref{nstr-g4} immediately implies that $W\cup Q$ is either a theta, or a pyramid or a prism in $G$, which violates $G\in \mathcal{C}^*$. This concludes the proof of Lemma~\ref{nstr-general}.
\end{proof}

Now we are ready to prove Theorem~\ref{starcutset}, which we restate:

\starcutset*

\begin{proof}

Suppose not. Let $D$ be a component of $G\setminus N[v]$ where $W\subseteq N[D]$. If $(W,v)$ is stranded with contour $(a_1,\ldots,a_k,b)$ for some $k\geq 3$, then $W[b^+,a_1^-]\cup W[a_k^+,b^-]\subseteq D$ and $\{a_2,\ldots, a_{k-1}\}\subseteq N(D)$. So there is a component of $G\setminus (N[v]\setminus \{a_2,\ldots, a_{k-1}))$ containing $D\cup \{a_2,\ldots, a_{k-1}\}$, which in turn contains $W\setminus \{a_1,a_k,b\}$. But this violates Lemma~\ref{strand-general}. It follows that $(W,v)$ is not stranded. Now, since $W\subseteq N[D]$ and $(W,v)$ has at least three sectors, $D$ is not $(W,v)$-local, a contradiction with Lemma~\ref{nstr-general}. This completes the proof of Theorem~\ref{starcutset}.
\end{proof}

\section{Hub-neighbors and  separators} \label{Hubs}

A vertex $v$ is a {\em hub} in $G$ if there is a wheel $(W,v)$ in $G$. If $(W, v)$ is a wheel, we say $v$ {\em is a hub for $W$}. 
We denote by $\Hub(G)$ the set of all hubs of $G$. We write 
$\deg_{\Hub(G)}(v)=|N_{\Hub(G)}(v)|$.
Let $v \in V(G)$ and let $D$ be a component of   $G \setminus N[v]$.
The {\em $v$-closure} of $D$ in $G$ is the set $N[D] \cup \{v\}$.
We denote the $v$-closure of $D$ in $G$  by $\cl_{v,G}(D)$.
Recall that the Ramsey number $R(t,s)$ is the minimum integer such that every graph on
at least $R(t,s)$ vertices contains either a clique of size $t$ or a stable set of size $s$.

The goal of this section is to show if $G$ is a graph in $\mathcal{C}$ and
$X$ is a  minimal separator or PMC in $G$, then the size of the
set $X \setminus \Hub(G)$ is small. In later sections we develop tools to
control the size of $X \cap \Hub(G)$ for certain sets $X$.

We need the following result from \cite{prismfree}.
\begin{lemma}\label{minimalconnected}
Let $x_1, x_2, x_3$ be three distinct vertices of a graph $G$. Assume that $H$ is a connected induced subgraph of $G \setminus \{x_1, x_2, x_3\}$ such that $V(H)$ contains at least one neighbor of each of $x_1$, $x_2$, $x_3$, and that $V(H)$ is minimal subject to inclusion. Then, one of the following holds:
\begin{enumerate}[(i)]
\item For some distinct $i,j,k \in  \{1,2,3\}$, there exists $P$ that is either a path from $x_i$ to $x_j$ or a hole containing the edge $x_ix_j$ such that
\begin{itemize}
\item $V(H) = V(P) \setminus \{x_i,x_j\}$; and
\item either $x_k$ has two non-adjacent neighbors in $H$ or $x_k$ has exactly two neighbors in $H$ and its neighbors in $H$ are adjacent.
\end{itemize}

\item There exists a vertex $a \in V(H)$ and three paths $P_1, P_2, P_3$, where $P_i$ is from $a$ to $x_i$, such that 
\begin{itemize}
\item $V(H) = (V(P_1) \cup V(P_2) \cup V(P_3)) \setminus \{x_1, x_2, x_3\}$;  
\item the sets $V(P_1) \setminus \{a\}$, $V(P_2) \setminus \{a\}$ and $V(P_3) \setminus \{a\}$ are pairwise disjoint; and
\item for distinct $i,j \in \{1,2,3\}$, there are no edges between $V(P_i) \setminus \{a\}$ and $V(P_j) \setminus \{a\}$, except possibly $x_ix_j$.
\end{itemize}

\item There exists a triangle $a_1a_2a_3$ in $H$ and three paths $P_1, P_2, P_3$, where $P_i$ is from $a_i$ to $x_i$, such that
\begin{itemize}
\item $V(H) = (V(P_1) \cup V(P_2) \cup V(P_3)) \setminus \{x_1, x_2, x_3\} $; 
\item the sets $V(P_1)$, $V(P_2)$ and $V(P_3)$ are pairwise disjoint; and
\item for distinct $i,j \in \{1,2,3\}$, there are no edges between $V(P_i)$ and $V(P_j)$, except $a_ia_j$ and possibly $x_ix_j$.
\end{itemize}
\end{enumerate}
\end{lemma}

We use Lemma~\ref{minimalconnected} in order to prove the following.

\begin{theorem} \label{smallminimal}
  Let $G \in \mathcal{C}_t$, let $X$ be a minimal separator of $G$,
  and let $Y \subseteq X \setminus \Hub(G)$ be stable. Then $|Y| \leq 2$.
  Consequently,  $|X \setminus \Hub(G)| \leq    R(t,3)$.
\end{theorem}

\begin{proof}
  We first prove the first assertion of the theorem.
Suppose $|Y| \geq 3$. Let $x_1,x_2,x_3 \in Y$.
Let $D_1,D_2$ be distinct full components for $X$.
  For $i \in {1,2}$, let $Z_i$ be a minimal connected subgraph of
  $D_i$ containing neighbors of $x_1,x_2,x_3$.
  We apply Lemma~\ref{minimalconnected} to
  $x_1,x_2,x_3$ and  $Z_1,Z_2$. If outcome (ii) holds for both $Z_1$ and $Z_2$, then $G$ contains a theta, a contradiction. If outcome (ii) holds for $Z_1$ and outcome (iii) holds for $Z_2$ (or vice versa), then $G$ contains a pyramid, a contradiction. If outcome (iii) holds for both $Z_1$ and $Z_2$, then $G$ contains a prism, a contradiction. Therefore, we may assume that outcome (i) holds for $Z_1$, and that $Z_1$ is a path from $x_1$ to $x_2$, where $x_3$ has at least two neighbors in $Z_1$. 
  Since $x_3$ is not a hub in $G$ and $G$ does not contain a pinched prism or a pyramid, it
  follows that: 
  
   \sta{\label{P2}
    For every path  $P_2$ from
        $x_1$ to $x_2$ with $P_2^* \subseteq D_2$, we have that $x_3$ is
        anticomplete to
        $P_2$.}

Suppose outcome (ii) holds for $Z_2$, so there exists a vertex $a \in Z_2$ and three paths $P_1, P_2, P_3$ in $Z_2$ from $a$ to $x_1, x_2, x_3$ as in (ii) of Lemma~\ref{minimalconnected}. By \eqref{P2}, it follows that $x_3$ is not adjacent to $a$. If $x_3$ has two non-adjacent neighbors in $Z_1$, then $G$ contains a theta between $x_3$ and $a$, a contradiction. Thus, $x_3$ has exactly two adjacent neighbors in $Z_1$, but now $G$ contains a pyramid with apex $a$ and base $x_3 \cup N_{Z_1}(x_3)$, a contradiction. Therefore, outcome (ii) does not hold for $Z_2$. 

Next, suppose outcome (iii) holds for $Z_2$, so there exists a triangle $a_1a_2a_3$ in $Z_2$ and three paths $P_1, P_2, P_3$ as in (iii) of Lemma~\ref{minimalconnected}. If $x_3$ has two non-adjacent neighbors in $Z_1$, then $G$ contains a pyramid with apex $x_3$ and base $a_1a_2a_3$, a contradiction. Thus, $x_3$ has exactly two adjacent neighbors in $Z_1$, but now $G$ contains a prism with triangles $a_1a_2a_3$ and $x_3 \cup N_{Z_1}(x_3)$, a contradiction. 

Consequently, outcome (i) holds for $Z_2$.
By \eqref{P2}, $Z_2$ is not a path from $x_1$ to $x_2$.
  Thus, we may assume that $Z_2$ is a path from $x_2$ to $x_3$, and that $x_1$ has at least two neighbors in $Z_2$. If $x_1$ has two non-adjacent neighbors in $Z_2$ and $x_3$ has two non-adjacent neighbors in $Z_1$, then $G$ contains a theta between $x_1$ and $x_3$, a contradiction. If $x_1$ has two non-adjacent neighbors in $Z_2$ and $x_3$ has exactly two adjacent neighbors in $Z_1$, then $G$ contains a pyramid with apex $x_1$ and base $x_3 \cup N_{Z_1}(x_3)$, a contradiction. If $x_1$ has exactly two adjacent neighbors in $Z_2$ and $x_3$ has exactly two adjacent neighbors in $Z_1$, then $G$ contains a prism with triangles $x_3 \cup N_{Z_1}(x_3)$ and $x_1 \cup N_{Z_2}(x_1)$, a contradiction. 
This proves the first assertion of the theorem.
The second assertion follows immediately since $G$ is $K_t$-free.
This proves Theorem~\ref{smallminimal}.
    \end{proof}

Theorem~\ref{smallminimal} has the following corollary that we will need later.

\begin{corollary} \label{smallattachments}
  Let $G \in \mathcal{C}_t$, let $v \in V(G)$ and let $D$ be a component of
  $G \setminus N[v]$. Let $Y \subseteq N(D) \setminus \Hub(G)$ be stable. Then, $|Y| \leq 2$.
  Consequently, $|N(D) \setminus \Hub(G)| \leq R(t,3)$ and
  $|N(D)| <  R(t,3)+\deg_{\Hub(\cl_{v,G}(D))}(v) \leq R(t,3)+\deg_{\Hub(G)}(v)$.
  \end{corollary}
\begin{proof}

Observe that $N(D)$ is a minimal separator in $\cl_{v, G}(D)$, so it follows from Theorem~\ref{smallminimal} that $|Y| \leq 2$ and $|N(D) \setminus \Hub(\cl_{v, G}(D))| < R(t,3)$.  We also have that $N(D) \cap \Hub(\cl_{v, G}(D)) \subseteq N_{\Hub(\cl_{v,G}(D))}(v) \subseteq \Hub(G)$. Therefore, $|N(D) \setminus \Hub(G)| < R(t,3)$ and $|N(D)| < R(t, 3) + \deg_{\Hub(\cl_{v, g}(D))}(v) \leq R(t, 3) + \deg_{\Hub(G)}(v)$.
\end{proof}

We finish  this section by proving an analogue of Theorem~\ref{smallminimal} for
PMCs instead of minimal separators.

\begin{theorem} \label{smallPMC}
  Let $G \in \mathcal{C}_t$, and let $X$ be a PMC of $G$.
  Let $Y \subseteq X \setminus \Hub(G)$ be stable. 
  Then $|Y| \leq 3$.   Consequently,  $|X \setminus \Hub(G)| \leq  R(t,4)$.
\end{theorem}

\begin{proof}
  Suppose $|Y| \geq 4$, let $y_1,y_2,y_3,y_4 \in Y$. By Theorem~\ref{thm:PMC_characterization} for every $i \neq j \in \{1, \dots, 4\}$
  there exists a component $D_{ij}$ of $G \setminus X$ such that
  $y_i, y_j \in N(D_{ij})$; let $P_{ij}$ be a path from $y_i$ to $y_j$ with interior in $D_{ij}$. Recall that by Theorem~\ref{prop:PMC_adhesions_are_seps}
  all $N(D_{ij})$ are minimal separators, and therefore by
  Theorem~\ref{smallminimal} $N(D_{ij}) \cap Y=\{y_i, y_j\}$. But now
  $\{y_1,y_2,y_3,y_4\}  \cup P_{12} \cup P_{23} \cup P_{13} \cup P_{14} \cup P_{24}$
  is a theta in $G$, a contradiction. This proves the first assertion of the
  theorem. The second assertion follows from the fact that $G$ is $K_t$-free.
  \end{proof}

\section{Extending tree decompositions of neighborhoods}
\label{sec:neighborhood_construction}

Next we prove a result that allows us to extend a tree decomposition of
the neighborhood of a vertex of $G$  to a tree decomposition of $G$.

Let $v \in V(G)$, and let $D_1, \dots, D_m$ be the components of
$G \setminus N[v]$. 
By Corollary~\ref{smallattachments}, for every $i \in \{1, \dots, m\}$, it holds that
$|N(D_i) \setminus \Hub(G)| \leq R(t,3)$.
Let $H$ be the graph obtained from $N(v) \setminus \Hub(G)$ by adding
a stable set of new vertices $\{d_1, \dots, d_m\}$
where $N_H(d_i)=N(D_i) \setminus \Hub(G)$ for every $i$. In other
words, $H$ is obtained from $G \setminus N_{\Hub(G)}[v]$
by contracting each $D_i$ to a corresponding kkvertex $d_i$. 

Throughout this section, we fix $G$ and $H$ as above. We first prove a structural result about $H$. 

\begin{theorem}
  \label{HinCt}
  If $G \in \mathcal{C}_t$, then $H \cup \{v\} \in \mathcal{C}_t$ and $H$ is
  wheel-free.
\end{theorem}

\begin{proof}

  For $I \subseteq \{1,\dots, m\}$ let $G_I$ be obtained from
  $G \setminus  N_{\Hub(G)}[v]$
by contracting each $D_i$ with $i \in I$ to a vertex $d_i$.
Then $H=G_{\{1, \dots, m\}}$ and $G \setminus \{v\} =G_{\emptyset}$.

We prove by induction on $|I|$
that $G_{I} \cup \{v\} \in \mathcal{C}_t$ and that
$\Hub(G_{I} \cup \{v\}) \cap (N(v) \cup \bigcup_{i \in I} \{d_i\}) = \emptyset.$ 

Since $N_{G_{I}}(d_i) \subseteq N_G(v)$ for every $i \in I$,
it follows that $\omega(G_I \cup \{v\}) \leq w(G) <t$.

Now suppose $G_{I} \cup \{v\}$ contains a generalized prism, theta, or pyramid
$Q$.
We may assume that $1 \in I$ and $d_1 \in Q$. It follows that $d_1$ has at least two neighbors. Let $I'=I \setminus \{1\}$.
Inductively, $G_{I'} \cup \{v\} \in \mathcal{C}_t$.
If $d_1$ has exactly two  neighbors, say $a$ and $b$,  in $Q$, then we can replace
$d_1$ by a path from $a$ to $b$ with interior in $D_1$ to get
a configuration of the same type as $Q$ in $G_{I'} \cup \{v\}$, a contradiction.
Thus, $d_1$ has at least three neighbors in $Q$.
Note that inductively, $N(v) \cap \Hub(G_{I'} \cup \{v\})=\emptyset$, so by Corollary~\ref{smallattachments}, $N(d_1)$ does not contain a stable set of size
three. Consequently, one of the following holds (with notation as in the definitions of the corresponding graphs):
\begin{enumerate}
\item $Q$ is a pyramid or a prism and $d_1=b_1$, or $Q$ is a pinched prism and $d_1 \in N(b_1) \cap Q$; or
\item $Q$ is a pinched prism with center $d_1$.
  \end{enumerate}
  
\sta{\label{not-first} The first alternative does not hold.}

Suppose the first alternative holds, and so
$N(d_1)\cap Q =\{x_1,x_2,x_3\}$, where $x_2$ is adjacent to $x_3$, and there are no
other edges with both ends in $\{x_1,x_2,x_3\}$.
Let $Z$ be a minimal connected subgraph of $D_1$ containing a neighbor of
each of $x_1,x_2,x_3$. Then $Q'=(Q \setminus \{d_1\}) \cup Z$ is an induced
subgraph of $G_{I'} \cup \{v\}$, and therefore $Q'$ is not a generalized prism,
a pyramid or a theta.

We apply  Lemma~\ref{minimalconnected} to $x_1, x_2, x_3$ and $Z$. Suppose outcome (i) of Lemma~\ref{minimalconnected} holds. Suppose first that $P$ is a hole containing the edge $x_2x_3$. If $x_1$ has two non-adjacent neighbors in $P$, then $\{v, x_1\} \cup P$ contains a pyramid with apex $x_1$ and base $vx_2x_3$, a contradiction. So $x_1$ has exactly two adjacent neighbors in $P$, and now $\{v, x_1\} \cup P$ is a prism with triangles $x_1 \cup (N(x_1) \cap P)$ and $vx_2x_3$, a contradiction. Thus, we may assume that $P$ is a path from $x_1$ to $x_2$. Suppose first that $Q$ is a pyramid with apex $a'$.
Since $x_3$ has at least two neighbors in $P \setminus \{x_2\}$, it follows that $x_3$ has two non-adjacent neighbors in $P$. Therefore, $x_3$ has a neighbor in $P$ non-adjacent to $x_2$. If $x_3$ is non-adjacent to $a'$, then $Q'$ is a theta from $x_3$ to $a'$, a contradiction, so $x_3$ is adjacent to $a'$. Now, $x_3$ is a hub for the hole given by $Q' \setminus x_3$, 
a contradiction. Therefore, $Q$ is a prism or a pinched prism. Let $x_1'x_2'x_3'$ be the triangle of $Q$ not containing $d_1$, where possibly $x_3 = x_3'$ or $x_2 = x_2'$, and $Q$ contains paths from $x_i$ to $x_i'$ for $1 \leq i \leq 3$. If $x_3 = x_3'$, then $x_3$ is a hub for the hole given by $Q' \setminus \{x_3\}$, a contradiction, so $x_3 \neq x_3'$. But now $Q'$ contains a pyramid from $x_3$ to $x_1'x_2'x_3'$, a contradiction. Consequently, outcome (i) of Lemma~\ref{minimalconnected} does not hold. 

Next, suppose outcome (ii) of Lemma~\ref{minimalconnected} holds, so there exists a vertex $a \in Z$ and three paths $P_1, P_2, P_3$ from $a$ to $x_1, x_2, x_3$, respectively, such that $Z = P_1 \cup P_2 \cup P_3 \setminus \{x_1, x_2, x_3\}$ and $P_1 \setminus \{a\}$, $P_2 \setminus \{a\}$, and $P_3 \setminus \{a\}$ are pairwise disjoint and anticomplete to each other, except for the edge $x_2x_3$. 
We first consider the case in which both $P_2$ and $P_3$ have length one. Now $Q'$ is isomorphic to the graph obtained from $Q$ by subdividing the edge $d_1 x_1$, which is not in a triangle of $Q$. This implies that if $Q$ is a pyramid, prism, or pinched prism, then so is $Q'$, a contradiction.  
It follows that at least one of $P_2$, $P_3$ has length more than one, so $Z \cup \{v, x_1, x_2, x_3\}$ is a pyramid, a contradiction.
Consequently, outcome (ii) of Lemma~\ref{minimalconnected} does not hold. 

It follows that outcome (iii) of Lemma~\ref{minimalconnected} holds, so there exists a triangle $a_1a_2a_3 \in Z$ and paths $P_1, P_2, P_3$ as in outcome (iii) of Lemma~\ref{minimalconnected}. But now $\{v, x_1, x_2, x_3\} \cup Z$ is a prism with triangles $vx_1x_2$ and $a_1a_2a_3$, a contradiction. Thus, the first alternative does not hold. This proves \eqref{not-first}. \\

Next we prove: 

\sta{\label{not-second} The second alternative does not hold.} 

Let $Q$ be a pinched prism with $d_1 = b_1$. Let $p, q, r, s$ be the neighbors of $d_1$ in $Q$, with $p$ adjacent to $q$ and $r$ adjacent to $s$. Let $z \in D_1$. 
If $z$ has two non-adjacent neighbors in $\{p,q,r,s\}$, then $(Q \setminus \{d_1\}) \cup \{z\}$ is a pinched prism, pyramid, or theta in $G_{I'} \cup \{v\}$ (depending on whether $z$ has four, three, or two neighbors in $\{p, q, r, s\}$), a contradiction. Therefore, $N(z) \cap \{p,q,r,s\}$ is either a subset of $\{p,q\}$ or a subset of $\{r, s\}$ for all $z \in D_1$. 

Now let $R$ be a shortest path in $D_1$ with ends $u, w$, say, such that $u$ has a neighbor in $\{p,q\}$ and $w$ has a neighbor in $\{r, s\}$. Then no vertex in $R^*$ has a neighbor in $\{p, q, r, s\}$; consequently $(Q \setminus \{d_1\}) \cup R$ is a prism, pyramid, or theta in $G_{I'} \cup \{v\}$ (depending on whether $|N(\{u, w\}) \cap \{p, q, r, s\}|$ equals four, three, or two), a contradiction. 
This proves \eqref{not-second}. \\

Therefore, $G_I \cup \{v\}$ does not contain a generalized prism, theta, or pyramid. This proves that $G_I \cup \{v\} \in \mathcal{C}_t$. By Corollary~\ref{smallattachments}, $N(d_i)$ does not contain a stable set of size three, so it follows that $d_i \not \in \Hub(G_I \cup \{v\})$ for all $i \in I$. It remains to show that $N(v) \cap \Hub(G_I \cup \{v\}) = \emptyset$. Suppose $(W,x)$ is a wheel in $G_I \cup \{v\}$,  where $x \in N(v)$. Inductively,
$d_1 \in W$. Let $r,s$ be the neighbors of
$d_1$ in $W$, and let $W'$ be the hole obtained from $W$ by replacing
$d_1$ by a  path $P$ from $r$ to $s$ with interior in $D_1$. If $d_1$ is non-adjacent to $x$, then $(W',x)$ is a wheel in $G_{I'}$,
a contradiction. This proves that $d_1$ is adjacent to $x$.
Since $\{x,r,s\}$ is not a stable set in $N(D_1)$, we
may assume that $x$ is adjacent to $r$. This implies that
if two vertices of $W \setminus N(x)$ belong to different
sectors of $(W,x)$, then they also belong to different sectors of
$(W',x)$. Now it follows that $(W',x)$ is a wheel in
$G_{I'}$, a contradiction.
This proves that $N(v) \cap \Hub(G_I \cup \{v\})=\emptyset$, and completes the
proof of Theorem~\ref{HinCt}.
\end{proof}

Let $(T_0,\chi_0)$ be a
    tree decomposition of $H$, and for $i \in \{1, \dots, m\}$,
    let $(T_i,\chi_i)$ be a tree decomposition of $D_i$.
    For $i \in \{1, \dots, m\}$ let $v(i) \in T_i$ be some vertex such that that
    $d_i \in \chi_0(v(i))$.
        Let $T$ be a tree obtained from the union of $T_0,T_1, \dots, T_m$
    by adding, for every $i>0$,  a unique  edge from some vertex of $T_i$ to $v(i)$.
    For $u \in T$ , let $\chi(u)$ be defined as follows.
\begin{itemize}
    \item  If $u \in V(T_0)$, let
    $$\chi(u)=(\chi_0(u) \setminus \{d_1, \dots, d_m\})  \cup N_{\Hub(G)}(v) \cup \{v\} \cup \bigcup_{d_i \in \chi_0(u)}N_H(d_i).$$
    \item     If $u \in V(T_i)$ for $i \in \{1, \dots, m\}$,  let
    $$\chi(u)=\chi_i(u) \cup N_H(d_i) \cup N_{\Hub(G)}(v).$$
\end{itemize}

    We  prove the following:

    \begin{theorem}
      \label{extendnbrhood}
      With the notation as above, $(T,\chi)$ is a tree decomposition of $G$.
          \end{theorem}

    \begin{proof}
     Since $T$ is obtained by adding a single edge from $T_0$ to each of
     the trees $T_1, \dots, T_m$, it follows that $T$ is a tree.
     Clearly every vertex of $G$ is in $\chi(u)$ for some $u \in V(T)$.
     Next we check that for every edge $xy$ of $G$ there exists $u \in V(T)$
     such that $x,y \in \chi(u)$. This is clear if  $x,y \in N[v]$,  and if
     $x,y \in D_i$. 
   Thus we may assume that $x \in N(v)$ and
     $y \in D_1$, say. Now $x \in N(D_1) \cup N_{\Hub(G)}(v)$, and therefore
      $x \in \chi(u)$ for every $u \in V(T_1)$. Let $u \in V(T_1)$ such that $y \in \chi_i(u) \subseteq \chi(u)$; then $x,y \in \chi(u)$ as required.

   Finally we show that $\chi^{-1}(x)$ is a tree for every $x \in V(G)$.
 Let $x \in V(G)$.  For $i \in \{0, \dots, m\}$ define $F_{i}(x)=\{u \in V(T_i) \text{ such that }x \in \chi(u)\}$. Let $F(x)=F_0(x) \cup \bigcup_{i=1}^m(F_i(x))$. We  need to show that for every $x \in X$ we have that
     $T[F(x)]$ is connected.
     If $x \in N_{\Hub(G)}(v)$, then $F(x)=V(T)$, and the claim holds.
     If $x=v$, then $F(x)=V(T_0)$, and again the claim holds.
     Thus we may assume that $x \in V(H) \cup \bigcup_{i=1}^m D_i$. 
     If $F_0(x)=\emptyset$, then there exists a unique $i$ such that $x \in D_i$; therefore $F(x)=F_i(x)$, and $T[F(x)]$ is connected because $T_i[F_i(x)]$ is connected.
     Thus we may assume that $x \in V(H) \cap N(v)$. Let $I \subseteq \{1, \dots, m\}$ be the set of all $i$ such that $x \in N(D_i) = N(d_i)$. It follows that $F_i(x) = \emptyset$ for all $i \in \{1, \dots, m\} \setminus I$.
       
     First we observe  that $x \in \chi(v(i))$ for every $i \in I$, 
     and so there is an edge between
     $T_i[F_i(x)] = T_i$ and
     $T_0[F_0(x)]$. Since $F_i(x)=V(T_i)$ is connected, it is enough to prove that      $T_0[F_0(x)]$ is connected.

 Observe that
 $$F_0(x)=\chi_0^{-1}(x) \cup \bigcup_{i \in I} \chi_0^{-1}(d_i).$$
 Since $xd_i \in E(H)$ for every $i \in I$,  it follows
       that $\chi_0^{-1}(x) \cap \chi_0^{-1}(d_i) \neq \emptyset$ 
       for every $i \in I$. Since
       $T_0[\chi_0^{-1}(u)]$ is a tree for every $u \in V(H)$,
       it follows that $T_0[F_0(x)]$ is connected, as required. This
       proves Theorem~\ref{extendnbrhood}.

     \end{proof}

\section{Stable sets  of unbalanced hubs} \label{sec:centralbag}

Let $G$ be a graph with $|V(G)| = n$. We say that $v \in V(G)$ is {\em balanced} if
$|D| \leq \frac{n}{2}$ for every component $D$ of $G \setminus N[v]$. If $v$ is not balanced, then $v$ is {\em unbalanced}.

In this section, we describe a way to decompose  graphs in $\mathcal{C}_t$ using stable sets of unbalanced hubs. Recall that $\mathcal{C}$ is the class of (theta, pyramid, generalized prism)-free graphs. We begin with the following structural result.

    \begin{lemma} \label{nohub}
      Let $G$ be a cube-free graph in $\mathcal{C}$, and suppose that 
      $v$ is a hub of $G$. Let $D$ be a component of $G\setminus N[v]$.
      Then $v$ is not a hub in
      $\cl_{v,G}(D)$.
    \end{lemma}

    \begin{proof}
      Suppose $v$ is a hub in $\cl_{v,G}(v)$. Now, Theorem~\ref{starcutset} 
      applied to an optimal wheel with hub $v$ in $\cl_{v,G}(D)$ implies that either $\cl_{v,G} \setminus N[v]$
      is not connected, or some vertex of $N(v) \cap \cl_{v,G}(v)$  has no
      neighbor in $\cl_{v,G} \setminus N[v]$. But $\cl_{v,G} \setminus N[v]=D$,
and  $N(v) \cap \cl_{v,G}(v)=N(D)$
      a contradiction.
      This proves Lemma~\ref{nohub}.
    \end{proof}

    Let $G$ be an $n$-vertex graph, and let $v$ be an unbalanced vertex of $G$. Then there exists a unique
    component $D$ of $G \setminus N[v]$ with $|D|>\frac{n}{2}$. Write
    $B(v)=D$, $C(v)=\{v\} \cup N(D)$ and
    $A(v)=V(G) \setminus (B(v) \cup C(v))$. We call
    $(A(v),C(v),B(v))$ the {\em canonical star separation corresponding to
      $v$}. (Note that at this point we do not require that $A(v)$ is
    non-empty.)

    Following \cite{prismfree}, we say
     that two unbalanced vertices $u, v \in V(G)$ are {\em star twins} if $B(u)= B(v)$, $C(u) \setminus \{u\} = C(v) \setminus \{v\}$, and
    $A(u) \cup \{u\} = A(v) \cup \{v\}$. Note that every two of these conditions imply the third.

Let $S$ be a stable set of unbalanced vertices.
     Let $\mathcal{O}$ be a fixed total  ordering of $S$. Let $\leq_A$ be a relation on $S$ defined as follows (we think of $\leq_A$ as a ``partial order by the $A$-sides): 
\begin{equation*}
\hspace{2.5cm}
x \leq_A y \ \ \ \text{ if} \ \ \  
\begin{cases} x = y, \text{ or} \\ 
\text{$x$ and $y$ are star twins and $\mathcal{O}(x) < \mathcal{O}(y)$, or}\\ 
\text{$x$ and $y$ are not star twins and } y \in A(x).\\
\end{cases}
\end{equation*} 
Note that if $x \leq_A y$, then either $x = y$, or $y \in A(x).$

We start with the following three results:
  \begin{lemma} \label{shields}
If $y \in A(x)$, then  $A(y) \cup \{y\} \subseteq A(x) \cup \{x\}$.
\end{lemma}

  \begin{proof}
    Let $n = |V(G)|$.   Since $C(y)\subseteq N[y]$ and $y$ is anticomplete to $B(x)$, we have $B(x) \subseteq G \setminus N[y]$. Since $|B(x)|, |B(y)| > n/2$, it follows that $B(x) \cap B(y) \neq \emptyset$. Since $B(x)$ is connected and contains no vertex in $N(B(y)) \subseteq N(y)$, it follows that $B(x) \subseteq B(y)$. Let $a \in C(x)\setminus \{x\}$. Then $a$ has a neighbor in $B(x)$ and thus in $B(y)$. If $a \in N[y]$, then $a \in C(y)$. If $a \not \in N[y]$, then $a \in B(y)$. It follows that $C(x) \setminus \{x\} \subseteq C(y) \cup B(y)$.  But now  $A(y) \setminus \{x\} \subseteq A(x)$, as required. This
  proves Lemma~\ref{shields}.
  \end{proof}
  
\begin{lemma} \label{order}
      $\leq_A$ is a partial order on $S$. 
\end{lemma}
\begin{proof}
We first prove the following claim: 

\sta{\label{startwinlemma} If $x$ and $y$ satisfy $A(y) \cup \{y\} = A(x) \cup \{x\}$, then $x$ and $y$ are star twins.}

 Since $x \in A(y)$, it follows that $N[x] \subseteq A(y) \cup C(y)$; so $C(x) \setminus \{x\} = N(x) \cap C(x) = N(x) \setminus A(x) = N(x) \setminus A(y) = N(x) \cap C(y) \subseteq C(y) \setminus \{y\}$ (where we used that $y \not\in N(x)$), which shows that $C(x) \setminus \{x\} \subseteq C(y) \setminus \{y\}$; by symmetry, it follows that $C(x) \setminus \{x\} = C(y) \setminus \{y\}$. This proves \eqref{startwinlemma}.\vsp

We show that $\leq_A$ is reflexive, antisymmetric, and transitive. By definition, $\leq_A$ is reflexive. Suppose $x \leq_A y$ and $y \leq_A x$ for some $x, y \in S$ with $x \neq y$. Since $x \leq_A y$, it follows that $y \in A(x)$, and since $y \leq_A x$, it follows that $x \in A(y)$. From Lemma~\ref{shields}, it follows that $A(y) \cup \{y\} = A(x) \cup \{x\}$.
Thus, by \eqref{startwinlemma}, $x$ and $y$ are star twins. But $x \leq_A y$ implies that $\mathcal{O}(x) < \mathcal{O}(y)$, and $y \leq_A x$ implies that $\mathcal{O}(y) < \mathcal{O}(x)$, a contradiction. Therefore, $\leq_A$ is antisymmetric. 

Next we prove transitivity.
Suppose $x \leq_A y$ and $y \leq_A z$ for distinct $x, y, z \in V(G)$.
We need to show that $x \leq_A z$.
Since $y \leq_A z$, it follows that $z \in A(y)$. Since $x \leq_A y$, it follows that $y \in A(x)$.
By Lemma~\ref{shields}, it follows that $A(z) \cup \{z\} \subseteq A(y) \cup \{y\} \subseteq A(x) \cup \{x\}$; so $z \in A(x)$. If $x$ and $z$ are not star twins, then $x \leq_A z$, so we may assume that $x$ and $z$ are star twins. 

Since $x$ and $z$ are star twins, $A(x) \cup \{x\} = A(z) \cup \{z\}$. But now $A(z) \cup \{z\} = A(y) \cup \{y\} = A(x) \cup \{x\}$, so $x, y, z$ are all pairwise star twins by \eqref{startwinlemma}. It follows that $\mathcal{O}(x) < \mathcal{O}(y) < \mathcal{O}(z)$, so $x \leq_A z$, as required.  
This proves Lemma~\ref{order}.
\end{proof}

  Let $G$ be a graph, and let $S \subseteq V(G)$ be a stable set of unbalanced vertices of $G$. Let $\Core(S)$ be a the set of all $\leq_A$-minimal elements of $S$. 

  \begin{lemma}\label{looselylaminar}
    Let $S$ be as above, and let $u,v \in \Core(S)$. Then
    $A(u) \cap C(v)=C(u) \cap A(v)=\emptyset$.
  \end{lemma}

  \begin{proof}
    Suppose $x \in A(v) \cap C(u)$. Since $x \in C(u)$, it follows that
    $u \in A(v) \cup C(v)$. Since $u$ is non-adjacent to $v$, we have that
    $u \not \in C(v)$, and therefore $u \in A(v)$. 
    Since $u,v \in \Core(S)$,
    we have that $v \not \leq_A u$, and therefore $u$ and $v$ are star twins.
    But then again $u$ and $v$ are comparable in $\leq_A$ since they
    are comparable in $\mathcal{O}$, a contradiction. This proves Lemma~\ref{looselylaminar}.
\end{proof}
    
  Define $\beta(S)= \bigcap_{v\in \Core(S)}(B(v) \cup C(v))$.
We call $\beta(S)$ the {\em central bag} for $S$. Next, we prove several properties of central bags. 
  
  \begin{theorem}\label{centralbag}
    Let $G,S$ be as above with $G \in \mathcal{C}_t$. The following hold:
    \begin{enumerate}
    \item For every $v \in \Core(S)$ we have $C(v) \subseteq \beta(S)$. \label{central-1}
    \item For every $v \in \Core(S)$ we have that
      $\deg_{\beta(S)}(v) < R(t,3)+\deg_{\Hub(G)}(v)$.\label{central-2}
  \item For every component $D$ of $G \setminus \beta(S)$, there exists $v \in \Core(S)$ such that $D \subseteq A(v)$. Further, if $D$ is a component of $G \setminus \beta(S)$ and $v \in \Core(S)$ such that $D \subseteq A(v)$, then $N(D) \subseteq C(v)$. \label{central-3}
  \item If $G$ is cube-free, then $S \cap \Hub(\beta(S))=\emptyset$.\label{central-5}
    \end{enumerate}
    \end{theorem}

  \begin{proof}
    \eqref{central-1} is immediate
    from Lemma~\ref{looselylaminar}; and \eqref{central-2} follows immediately
    from Corollary~\ref{smallattachments}.

    Next we prove  \eqref{central-3}. Let $D$ be a component of $G \setminus \beta(S)$. Since  $G \setminus \beta(S)=\bigcup_{v \in \Core(S)}A(v)$, there exists
    $v \in \Core(S)$ such that $D \cap A(v) \neq \emptyset$.
    If $D \setminus A(v) \neq \emptyset$, then, since $D$ is connected, it follows that $D \cap N(A(v)) \neq \emptyset$; but then $D \cap C(v) \neq \emptyset$, contrary to \eqref{central-1}. Since $N(D) \subseteq \beta(S)$ and $N(D) \subseteq A(v) \cup C(v)$, it follows that $N(D) \subseteq C(v)$. 
    This proves \eqref{central-3}.


To prove \eqref{central-5}, let $u \in S \cap \Hub(\beta(S))$.
Then, by Lemma~\ref{nohub}, it follows that $\beta(S) \not \subseteq \cl_{u,G}(B(u))$.
Since $\cl_{u,G}(B(u))=B(u) \cup C(u)$, it follows that $u \not \in \Core(S)$.
But then $u \in A(v)$ for some $v \in \Core(S)$, and so $u \not \in \beta(S)$,
a contradiction. This proves \eqref{central-5} and completes the proof of
Theorem~\ref{centralbag}.
  \end{proof}

  We finish this section with a theorem that allows us to transform a
  tree decomposition of $\beta(S)$ into a tree decomposition of $G$.
   
    Let $G,S$ be as above with $G \in \mathcal{C}_t$ connected and cube-free, and let $D_1, \dots, D_m$ be the components of
    $G \setminus \beta(S)$. For $i \in \{1,\dots, m\}$ let
    $r(D_i)$ be the  $\mathcal{O}$-minimal vertex of $S$ such that
    $D_i \subseteq A(v)$ (such $v$ exists by Theorem~\ref{centralbag}\eqref{central-3}).

    Let $(T_{\beta},\chi_{\beta})$ be a
    tree decomposition of $\beta(S)$, and for $i \in \{1, \dots, m\}$
    let $(T_i,\chi_i)$ be a tree decomposition of $D_i$.
    Let $T$ a the tree obtained from the union of $T_{\beta},T_1, \dots, T_m$
    by adding, for each $i \in \{1, \dots, m\}$  a unique  edge from some vertex of $T_i$ to some vertex
    $v \in V(T_{\beta})$ such that $r(D_i) \in \chi_{\beta}(v)$.
    For $u \in T$ , let $\chi(u)$ be defined as follows.
\begin{itemize}
    \item  If $u \in V(T_\beta)$, let
    $$\chi(u)=\chi_{\beta}(u) \cup \bigcup_{v \in \Core(S) \cap \chi_{\beta}(u)}C(v).$$
    \item If $u \in V(T_i)$ for $i \in \{1, \dots, m\}$, let
    $$\chi(u)=\chi_i(u) \cup C(r(D_i)).$$
\end{itemize}

    \begin{theorem}
      \label{extendtree}
      With the notation as above, $(T,\chi)$ is a tree decomposition of $G$.
          \end{theorem}

   \begin{proof}
     Since $T$ is obtained by adding a single edge from $T_{\beta}$ to each of
     the trees $T_1, \dots, T_m$, it follows that $T$ is a tree.
     Clearly, every vertex of $G$ is in $\chi(v)$ for some $v \in V(T)$.
     Next we check that for every edge $xy$ of $G$ there exists $v \in V(T)$
     such that $x,y \in \chi(T)$. This is clear if  $x,y \in \beta(S)$ or if
     $x,y \in D_i$ for some $i$; thus we may assume that $x \in \beta(S)$ and
     $y \in D_1$, say. Then by Theorem~\ref{centralbag}\eqref{central-3}, $x \in C(r(D_1))$, and
     therefore $x \in \chi(v)$ for every $v \in V(T_i)$. Let $v \in V(T_i)$ such that $y \in \chi_i(v) \subseteq \chi(v)$; then $x,y \in \chi(v)$ as required.

 Let $x \in V(G)$. Define
     $F_{\beta}(x)=\{v \in V(T_{\beta}) \text{ such that }x \in \chi(v)\}$,
     and for $i \in \{1, \dots, m\}$ define $F_{i}(x)=\{v \in V(T_i) \text{ such that }x \in \chi(v)\}$. Let $F(x)=F_{\beta}(x) \cup \bigcup_{i=1}^m(F_i(x))$. We  need to show that for every $x \in X$ we have that
     $T[F(x)]$ is connected.
     If $x \not\in \beta(S)$, then $F_{\beta}(x)=\emptyset$, and there exists a unique $i$ such that $x \in D_i$; therefore $F(x)=F_i(x)$, and $T[F(x)]$ is connected because $T_i[F_i(x)]$ is connected.
     Thus we may assume that $x \in \beta(S)$. Let $I \subseteq \{1, \dots, m\}$
     be the set of all $i$ such that $x \in C(r(D_i))$. It follows that $F_i(x) = \emptyset$ for all $i \in \{1, \dots, m\} \setminus I$. 

     First we show that for every $i \in I$,  there is an edge between
     $T_i[F_i(x)] = T_i$ and     $T_{\beta}[F_{\beta}(x)]$. 
     Let $i \in I$;
     and let $v \in T_{\beta}$ be such that $v$ is adjacent to a vertex of
     $T_i$. Then $r(D_i) \in \chi_\beta(v)$, so $C(r(D_i)) \subseteq \chi(v)$, and therefore, $x \in \chi(v)$. Thus there is an edge between $T_i[F_i(x)]$ and
     $T_{\beta}[F_{\beta}(x)]$, as required.

     Now to show that $T[F(x)]$ is connected, it is enough to prove that
     $T_{\beta}[F_{\beta}(x)]$ is connected. Let
       $J(x)=\{v \in \Core(S)\text { such that } x \in C(v)\}$. Observe that
       $$F_{\beta}(x)=\chi_{\beta}^{-1}(x) \cup \bigcup_{v \in J(x)} \chi_{\beta}^{-1}(v).$$ By Theorem~\ref{centralbag}\eqref{central-1}, $xv \in E(\beta(S))$ for every $v \in J(x) \setminus \{x\}$, so it follows
       that $\chi_{\beta}^{-1}(x) \cap \chi_{\beta}^{-1}(v) \neq \emptyset$ 
       for every $v \in J(x)$. Since
       $T_{\beta}[\chi_{\beta}^{-1}(u)]$ is a tree for every $u \in V(G)$,
       it follows that $T_{\beta}[F_{\beta}(x)]$ is connected, as required. This
       proves Theorem~\ref{extendtree}.
      
     \end{proof}

   \section{A useful vertex partition} \label{sec:collections}
   For the remainder of the paper, all logarithms are taken in base 2.  The goal of this section is to prove the following.
\begin{theorem} \label{logncollections}
  Let $t \in \mathbb{N}$, and let $G$ be (theta, $K_t$)-free with $|V(G)|=n$.
  Let $\delta_t$ be as in  Theorem~\ref{degenerate}. Then, there is a partition
  $(S_1, \dots, S_k)$   of $V(G)$ with the following properties:
  \begin{enumerate}
  \item $k \leq \delta_t \log n$.
  \item $S_i$ is a stable set for every $i \in \{1, \dots, k\}$.
  \item For every $i \in \{1, \ldots, k\}$ and $v \in S_i$ we have
    $\deg_{G \setminus \bigcup_{j <i}S_j}(v)  \leq 4 \delta_t$. \label{hubsequence-3}
  \end{enumerate}
  \end{theorem}

\begin{proof}
  We start with the following.
  
  \sta{\label{nover2} At least $\frac{n}{2}$ vertices of $G$ have degree at most
        $4\delta_t$.}
        
  Since by Theorem~\ref{degenerate} $G$ is $\delta_t$-degenerate, it follows that
  $|E(G)| \leq \delta_t n$.
  Let $X$ be the set of vertices of $G$ with degree at least $4\delta_t$.
  Then the number of edges incident with vertices of $X$ is at least
  $$\frac{1}{2} |X| 4 \delta_t \leq |E(G)| \leq \delta_t n$$
and therefore $|X| \leq \frac{n}{2}$. This proves \eqref{nover2}.
\\
\\
\sta{\label{lognsets}
   There is a partition $T_1, \ldots, T_m$ of $V(G)$
        with the following properties:
\begin{enumerate}
  \item $m \leq \log n$.
    \item For every $i \in \{1, \ldots, m\}$ and $v \in T_i$ we have
    $\deg_{G \setminus \bigcup_{j <i}T_j}(v)  \leq 4 \delta_t$.
  \end{enumerate}}
  The proof is by induction on $n$. By \eqref{nover2} there is a set $T_1 \subseteq V(G)$ with $|T_1|=\frac{n}{2}$ such that every vertex of $T_1$ has degree
  at most $4 \delta_t$ in $G$. Inductively, $V(G \setminus T_1)$ has a
  partition $T_2, \ldots, T_m$ satisfying the requirement of \eqref{lognsets}
  with $m-1 \leq \log {\frac{n}{2}}=\log n -1$. Now $T_1, \ldots, T_m$
  is a required partition of $V(G)$ and $m \leq \log n$. This proves \eqref{lognsets}.\vsp
  
  Since by Theorem~\ref{degenerate} $G$ is $\delta_t$-degenerate, it follows that
  for every $i \in \{1, \dots, m\}$ the graph $G[T_i]$ can be
  $\delta_t$-colored; let $S_1^i, \dots, S_{\delta_t}^i$ be such a coloring of
  $T_i$. Now
  $$S_1^1, \dots, S_{\delta_t}^1, \dots, S_1^m, \dots, S_{\delta_t}^m$$
  is a required partition of $V(G)$.
  This proves Theorem~\ref{logncollections}.
  \end{proof}

\section{Putting everything together} \label{sec:proof}

Let $G$ be a graph. A {\em hub-partition} of $G$ is a partition
$S_1, \ldots, S_k$ of $G[\Hub(G)]$ as in
Theorem~\ref{logncollections}; we call $k$ the {\em order} of the partition.
We call the {\em hub-dimension} of $G$ (denoting it by
$\hdim(G)$) the smallest $k$ such that $G$ as a hub partition of order $k$.

We can now state a strengthening of Theorem~\ref{mainthm} that we will prove by induction on $|V(G)|$ and $\hdim(G)$.

\begin{theorem}
  \label{diminduction}
  Let $G \in \mathcal{C}_t$ be a graph with $|V(G)|=n$. Then
  $\tw(G) \leq R(t, 4) + R(t,4)(4 \delta_t+R(t, 3))(\log n+\hdim(G))$.
\end{theorem}

By Theorem~\ref{logncollections}, $\hdim(G) \leq \delta_t \log n$ for every $G \in \mathcal{C}_t$, so Theorem~\ref{diminduction}
immediately implies Theorem~\ref{mainthm}. We now prove Theorem~\ref{diminduction}.

\begin{proof}
  
The proof is by induction on $\hdim(G)$, and for fixed
$\hdim(G)$ by induction on $|V(G)|$.
Let $G \in \mathcal{C}_t$.
By Theomrem~\ref{nocliquecutset} we may assume that $G$ has no clique cutset
(and thus in particular, $G$ is connected).
If $G$ contains the cube, then by Theorem~\ref{contain_cube}
$|V(G)|<9t$, so we may assume that $G$
does not contain a cube.
Let $S_1, \ldots, S_k$ be a hub-partition of $G$ with $k=\hdim(G)$.
Let $\beta_0=G$.
Having defined $\beta_i$, if
$\beta_{final}$ is not yet defined, proceed as follows.
If $\Hub(\beta_i)=\emptyset$, let $\beta_{final}=\beta_i$.
Now assume that $\Hub(\beta_i) \neq \emptyset$.
Let $S_{i+1}'=S_{i+1} \cap \Hub(\beta_i)$. (Recall that every vertex of
$S_{i+1}$ is a hub in $G$, whereas $S_{i+1}'$ is the set of all
the vertices of $S_{i+1}$ that remain hubs in $\beta_{i}$, and
so results of Section~\ref{sec:centralbag}  can be applied to $S_{i+1}'$ in
$\beta_i$.) If some vertex in
$S_{i+1}'$ is balanced (in $\beta_i$), let $\beta_{final}=\beta_i$.
In both cases define
$S_{final} = S_{i+1}'$, $depth_{final} = i$ and $\mathcal{P}_{final}=\bigcup_{j=1}^i S_i$.
Now we may assume that every vertex of $S_{i+1}'$ is
unbalanced (in $\beta_i$). 
Applying the construction described
in Section~\ref{sec:centralbag} to $S_{i+1}'$ and $\beta_i$,
let $\beta_{i+1}=\beta(S'_{i+1})$. Notice that, since we are working in $\beta_i$, canonical star separations for $S_{i+1}'$ are defined in $\beta_i$, and the relation $\leq_A$ and the set $\Core(S'_{i+1})$ in $\beta_i$ rely on the canonical star separations defined in $\beta_i$. By repeated applications of Theorem~\ref{centralbag}\eqref{central-5}, we deduce:

  \sta{\label{nofuturehubs}
    We have $S_j \cap \Hub(\beta_i)=\emptyset$ for every $j \leq i$.
        In particular, $\mathcal{P}_{final} \cap \Hub(\beta_{final})=\emptyset$.}
        
   By \eqref{nofuturehubs} and assertion \eqref{hubsequence-3} of Theorem~\ref{logncollections}, we have the following useful fact.  
   
   \sta{\label{hubdegree}
   For all $1 \leq i < k$ and for all $v \in S_{i+1}$, we have $\deg_{\Hub(\beta_i)}(v) \leq 4\delta_t$. }
   
  Next we show:
  
  \sta{\label{twbetafinal} We have $\tw(\beta_{final}) \leq  R(t, 4) + R(t,4)(4 \delta_t+R(t, 3))(\hdim(G)-depth_{final}+\log (|\beta_{final}|)).$}
  
  Suppose first that $\Hub(\beta_{final}) \neq \emptyset$.
  Then, there is a balanced vertex $v \in S_{final}$. By \eqref{nofuturehubs}, it follows that $\Hub(\beta_{final}) \subseteq \Hub(G) \setminus \mathcal{P}_{final}$, and so
  $\hdim(\beta_{final}) \leq k-depth_{final}$.
  Let $D_1, \hdots, D_m$ be the components of $\beta_{final} \setminus N[v]$. Since $v$ is balanced, it holds that $|D_i| \leq \frac{|\beta_{final}|}{2}$ for every $1 \leq i \leq m$.
  Let $H$ be the graph obtained from $\beta_{final} \setminus \{v\}$ by contracting
  each $D_i$ to a vertex $d_i$. By Theorem~\ref{HinCt}, we have $H \in \mathcal{C}_t$
  and $\Hub(H) = \emptyset$. Therefore, $\hdim(H) = 0$, and it follows inductively that  
  $$\tw(H) \leq R(t, 4) + R(t,4)(4 \delta_t+R(t, 3))\log (|V(H)|).$$

By Theorem~\ref{structured}, $H$ has a tree decomposition
$(T_0,\chi_0)$ of width $\tw(H)$,  where $\chi_0(v)$ is a PMC of $H$ for all $v \in V(T_0)$.
  Let $T$ be a tree decomposition of $\beta_{final}$ as in
  Theorem~\ref{extendnbrhood} (with $G=\beta_{final}$).
  We will show that $(T,\chi)$ has the required width.
  
  Let $u \in V(T)$; we will show an upper bound on $|\chi(u)|$.
  Suppose first that $u \in T_0$. By the construction in Section~\ref{sec:neighborhood_construction}, we have $\chi(u) = (\chi_0(u) \setminus \{d_1, \hdots, d_m\}) \cup N_{\Hub(\beta_{final})}(v) \cup \{v\}\cup \bigcup_{d_i \in \chi_0(u)} N_H(d_i)$. 
  Since  $\{d_1, \dots, d_m\} \cap \Hub(H) = \emptyset$, it follows
  from Theorem~\ref{smallPMC} that
  $|\chi_0(u) \cap \{d_1, \dots, d_m\}| \leq R(t,4)$.
  Next, by Corollary~\ref{smallattachments}, for every
  $i \in \{1, \dots, m\}$ we have that
  $|N_H(d_i)| \leq  R(t,3)$.
  Finally, by \eqref{hubdegree}, it follows that $\deg_{\Hub(\beta_{final})}(v) \leq 4 \delta_t$,
  so
  $$|\chi(u)| \leq |\chi_0(u)|+R(t,4)(R(t, 3)+4\delta_t).$$
  By the upper bound on $\tw(H)$, it holds that $|\chi_0(u)| \leq R(t, 4) + R(t, 4)(4\delta_t + R(t, 3))\log(|V(H)|)$, and so
  $$|\chi(u)| \leq R(t, 4) + R(t, 4)(4\delta_t + R(t, 3))(\log(|V(H)|) + 1).$$

  Since $|V(H)| \leq |\beta_{final}|$ and $\hdim(G) - depth_{final} \geq \hdim(\beta_{final}) \geq 1$, the required upper bound holds. Thus, we may assume that $u \in V(T_i)$. By the construction in Section~\ref{sec:neighborhood_construction}, we have $|\chi(u)| = |\chi_i(u)| + |N(d_i)| + |N_{\Hub(\beta_{final})}(v)|$. 
  By Corollary~\ref{smallattachments}, it holds that $|N(d_i)| \leq R(t, 3)$, and by \eqref{hubdegree}, it holds that $N_{\Hub(\beta_{final})}(v) \leq 4\delta_t$. Therefore,
  $|\chi(u)| \leq |\chi_i(u)|+4\delta_t+R(t, 3)$.
  Since $|D_i| \leq \frac{|\beta_{final}|}{2}$, it follows inductively that
    $$\tw(D_i) \leq R(t, 4)+ R(t,4)(4 \delta_t+R(t, 3))(\log (|D_i|)+\hdim(\beta_{final})) \leq $$
$$   R(t, 4) + R(t,4)(4 \delta_t+R(t, 3))(\log (|\beta_{final}|)-1+\hdim(\beta_{final})).$$
  Since $R(t,4)(4 \delta_t+R(t, 3)) \geq 4\delta_t+R(t, 3)$, the required upper
  bound follows.

Now we may assume that $\Hub(\beta_{final})=\emptyset$. By Theorem~\ref{smallPMC}, it follows that every PMC of $\beta_{final}$ has size at most $R(t, 4)$. Now, by Theorem~\ref{structured}, $\tw(\beta_{final}) \leq R(t, 4)$, as required. This proves \eqref{twbetafinal}. 
\\
\\
Now we prove by  induction on $s$ that for every $s \leq depth_{final}$, it holds that
$$\tw(\beta_s) \leq  R(t, 4) +R(t,4)(4 \delta_t+R(t, 3))(k-s+ \log (|\beta_s|)).$$
The assertion follows from \eqref{twbetafinal} for $s=depth_{final}$.
Now assume that we know the assertion for $\beta_s$; we prove it for
$\beta_{s-1}$.

  By Theorem~\ref{structured}, $\beta_s$ has a tree decomposition
  $(T_0,\chi_0)$ of width $\tw(\beta_s)$, where $\chi_0(v)$ is a PMC of $\beta_s$
  for every $v \in V(T_0)$.
  Let $T$ be a tree decomposition of $\beta_{s-1}$ as in
  Theorem~\ref{extendtree} (with $G=\beta_{s-1}$ and $\beta(S) =\beta_{s}$,
  where we use $(T_0,\chi_0)$ as $(T_{\beta}, \chi_{\beta})$).
  We will show that $(T,\chi)$ has the required width.
Let $D_1, \dots, D_m$ be the components of $\beta_{s-1} \setminus \beta_s$.
  
  Let $u \in V(T)$; we will show an upper bound on $|\chi(u)|$.
  Let $j \in \{0, \dots, m\}$ be such that $u \in T_j$. We show:
  
  \sta{\label{chij(u)} We have
  $$|\chi_j(u)| \leq R(t, 4) + R(t, 4)(4\delta_t + R(t, 3))(k - s + \log(|\beta_{s-1}|)).$$}
  
  If $j=0$, \eqref{chij(u)}  is true inductively since $|\beta_s| \leq |\beta_{s-1}|$.  Thus we may assume that $j \geq 1$.
  By \eqref{nofuturehubs}, $\hdim(\beta_{s-1}) \leq k-s+1$.
 By Theorem~\ref{centralbag}\eqref{central-3}, it holds that $|D_j| \leq \frac{|\beta_{s-1}|}{2}$, so  we deduce inductively that 
  $$|\chi_j(u)| \leq \tw(D_j) \leq R(t, 4) + R(t,4)(4 \delta_t+R(t, 3))(\log (|D_i|)+\hdim(\beta_{s-1})) \leq $$
$$   R(t, 4)+ R(t,4)(4 \delta_t+R(t, 3))(\log (|\beta_{s-1}|)-1+(k-s+1)).$$
        This proves \eqref{chij(u)}. \\

  Next, by \eqref{hubdegree} and by Theorem~\ref{centralbag}\eqref{central-2}, we have the following: 
  
  \sta{\label{C(v)_bounded}
      For every
  $v \in \Core(S_s)$ we have that $\deg_{\beta_s}(v) < 4\delta_t+R(t, 3)$.}

  Now we show:

\sta{\label{growbag}
       $|\chi(u)| \leq |\chi_j(u)|+R(t,4)(4\delta_t+R(t, 3))$.}

Suppose first that  $j>0$. Then, by the construction in Section~\ref{sec:centralbag}, we have $\chi(u) = \chi_j(u) \cup C(r(D_j))$. 
  Now by \eqref{C(v)_bounded}, it holds that
  $|\chi(u)| \leq |\chi_j(u)|+4\delta_t+R(t, 3)$.
Since $R(t,4)(4 \delta_t+R(t, 3)) \geq 4\delta_t+R(t, 3)$, the required upper
  bound follows.
  Thus we may assume that $j=0$.
  By the construction in Section~\ref{sec:centralbag}, we have $\chi(u) = \chi_0(u) \cup \bigcup_{v \in \Core(S_s) \cap \chi_0(u)} C(v)$.
  By \eqref{nofuturehubs}, it holds that
  $S_{s} \cap \Hub(\beta_s)=\emptyset$, so by Theorem~\ref{smallPMC} it follows that
  $\Core(S_s) \cap \chi_0(u) \leq R(t,4)$.
  It follows from \eqref{C(v)_bounded} that
  $|\chi(u)| \leq |\chi_0(u)|+R(t,4)(4\delta_t+R(t, 3))$.
This proves \eqref{growbag}.
\\
\\
Now the result follows immediately from \eqref{chij(u)} and \eqref{growbag}.
  This completes the proof of Theorem~\ref{diminduction}.

\end{proof}

\section{Algorithmic consequences}\label{algsec}

Finally, let us shed some light on the algorithmic significance of Theorem~\ref{mainthm}. It is well-known that a wide range of \textsf{NP}-hard problems are polynomial-time solvable on graphs of bounded treewidth, provided that a tree decomposition of bounded width is given. In particular, according to a celebrated theorem of Courcelle \cite{courcelle1990monadic}, for every graph property $\mathtt{P}$ which is definable in the monadic second-order logic, there exists a computable function $f$ such that given a graph $G$ and a tree decomposition of $G$ of width at most $k$, one can check in time $\mathcal{O}(f(k)|V(G)|)$ whether $G$ satisfies $\mathtt{P}$. The function $f$ can (and in fact `should', in some sense) be quite huge for some properties. But for a broad category of `local' problems, the standard dynamic programming on a tree decomposition of bounded width \cite{Bodlaender1988DynamicTreewidth} yields $f(k)=2^{\mathcal{O}(k)}$. Roughly speaking, \textit{local} problems, such as \textsc{Stable Set}, \textsc{Vertex Cover}, \textsc{Dominating Set} and \textsc{$k$-Coloring} (with fixed $k$), are those for which, while applying the dynamic programming to a tree decomposition of bounded width, for each bag it suffices to keep track of those subsets of the bag that belong to a partial solution (again, see \cite{Bodlaender1988DynamicTreewidth} for more details).  See also \cite{BODLAENDER201586, 6108160} for more information about problems attaining $f(k)=2^{\mathcal{O}(k)}$.

Note that while Courcelle's theorem guarantees a polynomial-time algorithm in graph classes of bounded treewidth, aligning with Theorem~\ref{mainthm}, the major advantage of $f$ being singly exponential (with linear exponent) in $k$ is that the corresponding problem becomes polynomial-time solvable in much richer classes of graphs, namely those with logarithmic treewidth. 

This leaves as the only concern to compute efficiently a tree decomposition of logarithmic width. The following $4$-approximation for computing a tree decomposition of minimum width due to Robertson and Seymour \cite{RS-GMXIII} is strong enough for our purposes.

\begin{theorem}[\cite{RS-GMXIII}, see also \cite{Marcinbook}]
There exists an algorithm which runs in time $\mathcal{O}(27^k\cdot k^2\cdot n^2)$ and given a graph $G$ and an integer $k$, the algorithm either correctly outputs $\tw(G)>k$ or computes a tree decomposition of $G$ of width at most $4k$.
\end{theorem}

From this, we immediately deduce:

\begin{theorem}\label{generalalg}
Let {\rm\texttt{P}} be a problem which admits an algorithm running in time $\mathcal{O}(2^{\mathcal{O}(k)}|V(G)|)$ on graphs $G$ with a given tree-decomposition of width at most $k$. Also, let $\mathcal{G}$ be a class of graphs for which there exists a constant $c=c(\mathcal{G})$ such that $\tw(G)\leq c\log(|V(G)|)$ for all $G\in \mathcal{G}$. Then {\rm\texttt{P}} is polynomial-time solvable in $\mathcal{G}$.
\end{theorem}

In view of Theorems~\ref{mainthm} and \ref{generalalg}, we conclude the following.

\begin{theorem}\label{algfinal}
Let $k\geq 1$ be fixed and {\rm\texttt{P}} be a problem which admits an algorithm running in time $\mathcal{O}(2^{\mathcal{O}(t)}|V(G)|)$ on graphs $G$ with a given tree-decomposition of width at most $k$. Then {\rm\texttt{P}} is polynomial-time solvable in $\mathcal{C}_t$. In particular, \textsc{Stable Set}, \textsc{Vertex Cover}, \textsc{Dominating Set} and \textsc{$q$-Coloring} (with fixed $q$) are all polynomial-time solvable in $\mathcal{C}_t$.
\end{theorem}
There is still one important problem to look into, and that is \textsc{Coloring}. By Theorem~\ref{degenerate}, for every fixed $t$, all graphs in $\mathcal{C}_t$ have chromatic number at most $\delta_t$. Also, for each fixed $k$, by Theorem~\ref{algfinal}, $k$-\textsc{Coloring} is polynomial-time solvable in $\mathcal{C}_t$. Now by solving  $k$-\textsc{Coloring} for every $k \in \{1, \dots, \delta_t\}$, we obtain a polynomial-time algorithm for \textsc{Coloring} in $\mathcal{C}_t$.

\section{Acknowledgments}
We are grateful to Cemil Dibek and Kristina Vu\v{s}kovi\'c for their involvement in the early stages of this work. We also thank Ron Aharoni, Eli Berger and Shira Zerbib for helpful discussions of possible strengthenings of Theorem~\ref{logncollections}. Finally, we thank Nicolas Trotignon for suggesting adding
Section~\ref{algsec} to the paper.


\begin{aicauthors}
\begin{authorinfo}[ta]
  Tara Abrishami\\
  Princeton University\\
  Princeton, New Jersey, USA\\
\end{authorinfo}
\begin{authorinfo}[mc]
  Maria Chudnovsky\\
Princeton University \\
Princeton, New Jersey, USA \\
\end{authorinfo}
\begin{authorinfo}[sh]
Sepehr Hajebi \\
Department of Combinatorics and Optimization \\
University of Waterloo \\
Waterloo, Ontario, Canada\\
\end{authorinfo}

\begin{authorinfo}[sophie]
Sophie Spirkl \\
Department of Combinatorics and Optimization \\
University of Waterloo \\
Waterloo, Ontario, Canada\\
\end{authorinfo}
\end{aicauthors}


\begin{thebibliography} {99}
  
  \bibitem{Aboulker2020OnGraphs}
 P. Aboulker, I. Adler, E. J. Kim, N. L. D. Sintiari, and N. Trotignon.
\newblock {``On the tree-width of even-hole-free graphs.''}
\newblock {\em arXiv:2008.05504\/} (2020).

\bibitem{wallpaper}
T. Abrishami, M. Chudnovsky, C. Dibek, S. Hajebi, P. Rz\k{a}\.{z}ewski, S. Spirkl, and K. Vu\v{s}kovi\'c. ``Induced subgraphs and tree decompositions II. Toward walls and their line graphs in graphs of bounded degree.'' {\em arXiv:2108.01162}, (2021). 

\bibitem{prismfree} T. Abrishami, M. Chudnovsky, C. Dibek and K. Vu\v{s}kovi\'c,
  ``Submodular functions and perfect graphs'', {\em arXiv:2110.00108}, (2021). 

\bibitem{ACV}
T. Abrishami, M. Chudnovsky, and K. Vu\v{s}kovi\'c. ``Induced subgraphs and tree decompositions I. Even-hole-free graphs of bounded degree.'' {\em arXiv:2009.01297}, (2020). 


\bibitem{Bodlaender1988DynamicTreewidth}
 H.~L. Bodlaender.
\newblock {``Dynamic programming on graphs with bounded treewidth.''}
\newblock Springer, Berlin, Heidelberg, (1988), pp.~105--118.

\bibitem{BODLAENDER201586}
  H.~L. Bodlaender, M. Cygan and S. Kratsch and J. Nederlof. ``Deterministic single exponential time algorithms for connectivity problems parameterized by treewidth.'' {\em Information and Computation} {\bf 243} (2015), 86-111.


\bibitem{cliquetw}
H. Bodlaender and A. Koster. ``Safe separators for treewidth.'' {\em Discrete Mathematics} {\bf 306}, 3 (2006), 337--350. 

    \bibitem{BouchitteT01} V. Bouchitt\'e, and I. Todinca.
``Treewidth and minimal fill-in: Grouping the minimal separators'',
{\em SIAM Journal on Computing} {\bf  31} (2001), 212--232.

\bibitem{courcelle1990monadic} B. Courcelle. ``The monadic second-order logic of graphs. I. Recognizable sets of finite graphs." {\em Information and computation} {\bf 85} (1990), 12--75.

\bibitem{6108160} M.Cygan, J. Nederlof, and M. Pilipczuk, M. Pilipczuk, J.M.M. Rooij, and J.O. Wojtaszczyk. ``Solving Connectivity Problems Parameterized by Treewidth in Single Exponential Time." {\em 2011 IEEE 52nd Annual Symposium on Foundations of Computer Science} (2011), 150--15.


\bibitem{diestel} 
R. Diestel. {\em Graph Theory}. Springer-Verlag, Electronic Edition, (2005). 

\bibitem{KuhnOsthus} D. Kuhn and D. Osthus.
  ``Induced subgraphs in $K_{s,s}$-free graphs of large average degree'',
  {\em Combinatorica} {\bf 24} (2004), 287--304.
  
  \bibitem{Marcinbook} M. Pilipczuk. ``Computing Tree Decompositions.'' {\em In: Fomin F., Kratsch S., van Leeuwen E. (eds) Treewidth, Kernels, and Algorithms. Lecture Notes in Computer Science}, vol 12160. Springer, Cham. (2020)  


  \bibitem{septotree} M. Pilipczuk, N.L.D. Sintiari, S. Thomass\'e and
    N. Trotignon. ``(Theta, triangle)-free and (even-hole, $K_4$)-free graphs.
    Part 2: bounds on treewidth'', {\em J. Graph Theory} \textbf{97} (4) (2021), 624–641. 
    
      \bibitem{RS-GMV}
N. Robertson and P.D.~Seymour. ``Graph minors. V. Excluding a planar graph.'' \textit{J. Combin. Theory Ser. B}, 41 (1) (1996), 92–114.
 
 
 \bibitem{RS-GMXIII} N. Robertson and P.D.~Seymour. ``Graph minors. XIII. The disjoint paths problem.''
{\em J. Comb. Theory Ser. B}, {\bf 63}(1) (1995), 65--110.


    
\bibitem{mainconj} N.L.D. Sintiari and N. Trotignon.
  ``(Theta, triangle)-free and (even-hole, $K_4$)-free graphs. Part 1: Layered wheels'', {\em  J. Graph Theory} \textbf{97} (4) (2021), 475-509. 



  
      
  \end{thebibliography}
\end{document}